\RequirePackage[l2tabu, orthodox]{nag} 

\documentclass{amsart}

\pdfoutput=1

\usepackage[utf8]{inputenc} 
\usepackage[colorlinks, linkcolor=blue,citecolor=blue, breaklinks=true]{hyperref} 
\usepackage[stretch=10, shrink=10]{microtype} 
\usepackage{todonotes} 
\usepackage{amscd, amssymb}  
\usepackage{commath} 
\usepackage{caption}
\usepackage{subcaption}
\usepackage{mathtools} 
\usepackage[initials,msc-links,nobysame]{amsrefs}[2007/10/22]

\usepackage{paralist,enumitem}

\setlist{listparindent=0pt,parsep=3pt}
\setdefaultenum{1.}{}{}{}

\newsavebox{\largestimage}

\newcommand{\TitleWithUrl}[1]{\IfEmptyBibField{doi}%
  {\IfEmptyBibField{url}{\textit{#1}}%
    {\IfEmptyBibField{eprint}{\href {\BibField{url}}{\textit{#1}}}{\textit{#1}}}%
    }%
  {\href {https://doi.org/\BibField{doi}}{\textit{#1}}}}
\renewcommand{\eprint}[1]{\IfEmptyBibField{url}{\url{#1}}%
  {\href {\BibField{url}}{#1}}}

\BibSpec{article}{%
    +{}  {\PrintAuthors}                {author}
    +{,} { \TitleWithUrl}               {title}
    +{.} { }                            {part}
    +{:} { \textit}                     {subtitle}
    +{,} { \PrintContributions}         {contribution}
    +{.} { \PrintPartials}              {partial}
    +{,} { }                            {journal}
    +{}  { \textbf}                     {volume}
    +{}  { \PrintDate}                {date}
    +{,} { \issuetext}                  {number}
    +{,} { \eprintpages}                {pages}
    +{,} { }                            {status}
    +{,} { available at arXiv:\eprint}        {eprint}
    +{}  { \parenthesize}               {language}
    +{}  { \PrintTranslation}           {translation}
    +{;} { \PrintReprint}               {reprint}
    +{.} { }                            {note}
    +{.} {}                             {transition}
    +{}  {\SentenceSpace \PrintReviews} {review}
}

\BibSpec{collection.article}{%
    +{}  {\PrintAuthors}                {author}
    +{,} { \TitleWithUrl}                     {title}
    +{.} { }                            {part}
    +{:} { \textit}                     {subtitle}
    +{,} { \PrintContributions}         {contribution}
    +{,} { \PrintConference}            {conference}
    +{}  {\PrintBook}                   {book}
    +{,} { }                            {booktitle}
    +{,} { \PrintDateB}                 {date}
    +{,} { pp.~}                        {pages}
    +{,} { }                            {status}
    +{,} { available at \eprint}        {eprint}
    +{}  { \parenthesize}               {language}
    +{}  { \PrintTranslation}           {translation}
    +{;} { \PrintReprint}               {reprint}
    +{.} { }                            {note}
    +{.} {}                             {transition}
    +{}  {\SentenceSpace \PrintReviews} {review}
}
\BibSpec{book}{%
    +{}  {\PrintPrimary}                {transition}
    +{,} { \TitleWithUrl}               {title}
    +{.} { }                            {part}
    +{:} { \textit}                     {subtitle}
    +{,} { \PrintEdition}               {edition}
    +{}  { \PrintEditorsB}              {editor}
    +{,} { \PrintTranslatorsC}          {translator}
    +{,} { \PrintContributions}         {contribution}
    +{,} { }                            {series}
    +{,} { \voltext}                    {volume}
    +{,} { }                            {publisher}
    +{,} { }                            {organization}
    +{,} { }                            {address}
    +{,} { \PrintDateB}                 {date}
    +{,} { }                            {status}
    +{}  { \parenthesize}               {language}
    +{}  { \PrintTranslation}           {translation}
    +{;} { \PrintReprint}               {reprint}
    +{.} { }                            {note}
    +{.} {}                             {transition}
    +{}  {\SentenceSpace \PrintReviews} {review}
}

\BibSpec{thesis}{%
    +{}  {\PrintAuthors}                {author}
    +{.} { \TitleWithUrl}               {title}
    +{:} { \textit}                     {subtitle}
    +{,} { \PrintThesisType}            {type}
    +{,} { }                            {organization}
    +{} { \PrintDate}                  {date}
    +{,} { }                            {address}
    +{,} { \eprint}                     {eprint}
    +{,} { }                            {status}
    +{}  { \parenthesize}               {language}
    +{}  { \PrintTranslation}           {translation}
    +{;} { \PrintReprint}               {reprint}
    +{.} { }                            {note}
    +{.} {}                             {transition}
    +{}  {\SentenceSpace \PrintReviews} {review}
}

\newtheorem{theorem}{Theorem}[section]
\newtheorem{lemma}[theorem]{Lemma}
\newtheorem{corollary}[theorem]{Corollary}
\newtheorem{proposition}[theorem]{Proposition}

\theoremstyle{definition}
\newtheorem{definition}[theorem]{Definition}

\theoremstyle{remark}
\newtheorem{remark}[theorem]{Remark}

\newtheorem{assumption}[theorem]{Assumption}

\numberwithin{equation}{section}

\newcommand{\hermtwo}{\mathrm{Herm}(2)}
\newcommand{\SLtwoC}{\mathrm{SL}(2,\mathbb{C})}
\newcommand{\SUoneohone}{\mathrm{SU}(1,0,1)}
\newcommand{\suoneohone}{\mathfrak{su}(1,0,1)}

\newcommand{\SOthreeone}{\mathrm{SO}(3,1)}

\DeclareMathOperator{\trace}{tr}
\DeclareMathOperator{\lcm}{lcm}
\let\Im\relax
\DeclareMathOperator{\Im}{Im}
\let\Re\relax
\DeclareMathOperator{\Re}{Re}

\title{Spinor representation in isotropic 3-space via Laguerre geometry}

\author{Joseph Cho}
\address[Joseph Cho]{Institute of Discrete Mathematics and Geometry, TU Wien, Wien, 1040, Austria}
\email{jcho@geometrie.tuwien.ac.at}

\author{Dami Lee}
\address[Dami Lee]{Department of Mathematics, Indiana University, Bloomington, IN, 47405, USA}
\email{damilee@indiana.edu}

\author{Wonjoo Lee}
\address[Wonjoo Lee]{Department of Mathematics, Korea University, Seoul, 02841, Republic of Korea}
\email{wontail123@korea.ac.kr}

\author{Seong-Deog Yang}
\address[Seong-Deog Yang]{Department of Mathematics, Korea University, Seoul, 02841, Republic of Korea}
\email{sdyang@korea.ac.kr}

\subjclass[2020]{Primary 53A10; Secondary 53B30.}
\keywords{Laguerre geometry, isotropic geometry, spinor representation, Weierstrass representation, Kenmotsu representation}

\begin{document}
\begin{abstract}
	We give a detailed description of the geometry of isotropic space, in parallel to those of Euclidean space within the realm of Laguerre geometry.
	After developing basic surface theory in isotropic space, we define spin transformations, directly leading to the spinor representation of conformal surfaces in isotropic space.
	As an application, we obtain the Weierstrass-type representation for zero mean curvature surfaces, and the Kenmotsu-type representation for constant mean curvature surfaces, allowing us to construct many explicit examples.
\end{abstract}

\maketitle

\section{Introduction}

The Erlanger program \cite{klein_vergleichende_1872} of Klein was epochal to the field of differential geometry for obtaining space form geometries via \emph{transformation groups}, also investigated in his two papers \cite{klein_ueber_1871, klein_ueber_1873}.
The central idea was to present \emph{projective geometry} as the commonground for various known space form geometries.
In particular, applying the ideas of Cayley \cite{cayley_sixth_1859}, Klein showed a way to define metrics in the realm of projective geometry by restricting to projective transformations that leave a certain quadric invariant.
This was referred to as the \emph{absolute} by Cayley.
By doing so, Klein unified Euclidean geometry with the various non-Euclidean geometries including hyperbolic geometry and elliptical geometry as subgeometries of projective geometry.

Researchers found other ways to choose the absolute; the geometry of isotropic plane is one of such \emph{Cayley-Klein geometries}, first studied in works such as \cite{beck_zur_1913, study_zur_1909, berwald_uber_1915}.
The study of isotropic $3$-space followed soon, most notably by Strubecker, covering the basics of isotropic geometry \cite{strubecker_beitrage_1938}, space curve theory \cite{strubecker_differentialgeometrie_1941}, and surface theory \cite{strubecker_differentialgeometrie_1942-1, strubecker_differentialgeometrie_1942, strubecker_differentialgeometrie_1944, strubecker_differentialgeometrie_1949}.
Many modern expositions on the topic of isotropic plane and isotropic space including \cite{yaglom_simple_1979, sachs_isotrope_1990} also follow such approach of Klein.
Recently, the geometry of isotropic space has gained interest, most notably for their applicability to architectural geometry \cite{pottmann_discrete_2007} (see also, for example, \cite{vouga_design_2012, kilian_material-minimizing_2017, millar_designing_2022, tellier_designing_2023}).

On the other hand, \emph{spin transformations} for conformal surfaces in Euclidean space were defined in \cite{kamberov_bonnet_1998} by relating two conformally equivalent surfaces via homotheties and rotations of the corresponding tangent planes.
Spin transformations were used to characterize \emph{Bonnet pairs} via isothermic surfaces, a class of surfaces that admit conformal curvature coordinates.
The \emph{quaternionic description} of rotation in Euclidean space proved to be central to obtaining the Dirac-type equation as the compatibility condition for spin transformations \cite{kamberov_bonnet_1998}, and deriving \emph{spinor representations} of conformal surfaces \cite{kusner_spinor_1995} via spin transformations \cite{friedrich_spinor_1998}.

In this paper, we propose \emph{Laguerre geometry} as a basis for understanding isotropic space and surface theory within, which in itself is closely related to the approach of Klein.
Laguerre geometry \cite{laguerre_sur_1881} is concerned with the set of transformations in Euclidean space that maps points (viewed as zero radius spheres) and spheres to points and spheres, and planes to planes. (For details, see \cite{blaschke_vorlesungen_1929, cecil_lie_2008}, for example.)
The corresponding Laguerre geometry of isotropic geometry was given in \cite{graf_zur_1936, pottmann_applications_1998, pottmann_laguerre_2009}; our primary motivation for using Laguerre geometry is to unify the approaches to Euclidean geometry and isotropic geometry.
Thus in Section~\ref{sect:two}, we show that the Laguerre geometric description of Euclidean space carries over to isotropic space by viewing both space forms as hyperplanes in the Minkowski $4$-space (see also \cite{pember_weierstrass-type_2020, pember_discrete_2022}).
In particular, we carefully review the notions of Laguerre geometry in Euclidean space in Section~\ref{subsect:twoone}, where the Euclidean notions will provide valuable intuition for the definitions of points, spheres, planes, and normals in isotropic space, which we introduce in Section~\ref{subsect:twotwo}.
In doing so, we also check that our definitions are in harmony with the classical definitions given in \cite{strubecker_beitrage_1938} or \cite{pottmann_discrete_2007}.
Such approach allows for the geometry of isotropic space to be understood in parallel to that of Euclidean space, without any prior knowledge of projective geometry.

Laguerre geometric approach to isotropic space turns out to be highly suitable for surface theoretic considerations, which is the center of our attention in Section~\ref{sect:three}.
By viewing surfaces as codimension two immersions in Minkowski $4$-space, we use the notion of lightcone Gauss maps of \cite{izumiya_lightcone_2004}, which we call \emph{lightlike Gauss maps} after \cite{pember_weierstrass-type_2020}, to define the second fundamental form of the surface.
This allows us to recover the structure equations of Gauss and Weingarten, given by Strubecker in \cite{strubecker_differentialgeometrie_1942} (see Section~\ref{subsect:threeone}).

Furthermore, the well-known isomorphism between Minkowski $4$-space and Hermitian matrices allows us to use the matrix group $\SLtwoC$ to obtain a quaternion-like description of rotations in isotropic $3$-space.
Thus, after converting the structure equations in terms of Hermitian matrices in Section~\ref{subsect:threeone}, we define \emph{spin transformations} of conformal surfaces in isotropic $3$-space in Definition~\ref{def:spin}, and obtain \emph{Dirac-type equation} as the compatibility condition in Theorem~\ref{thm:Dirac}.
We note here that since surfaces in isotropic $3$-space can be regarded as objects in Minkowski $4$-space, the spin transformation is a special case of the result in \cite{bayard_spinorial_2013} (see also \cite{aledo_marginally_2005}).

Using spin transformations, we then obtain the spinor representation of conformal surfaces in isotropic $3$-space in Theorem~\ref{thm:spinor}, a notion that can also be interpreted as Kenmotsu representation for surfaces with prescribed mean curvature \cite{kenmotsu_weierstrass_1979} (see Remark~\ref{rem:kenmotsu}).
As an application, we recover the Weierstrass-type representation for minimal surfaces \cite{strubecker_differentialgeometrie_1942} (see also \cite{sachs_isotrope_1990,pember_weierstrass-type_2020, seo_zero_2021, da_silva_holomorphic_2021}) in Theorem~\ref{thm:Weierstrass}.
We also obtain the Kenmotsu-type representation for non-zero constant mean curvature (cmc) surfaces in isotropic $3$-space in Theorem~\ref{thm:Kenmotsu}.
A standout feature of the Kenmotsu-type representation for cmc surfaces in isotropic $3$-space is that one can obtain many explicit examples of cmc surfaces using the representation; we demonstrate this with concrete examples given in Section~\ref{subsect:examples}.

\section{Geometry of isotropic \texorpdfstring{$3$}{3}-space}\label{sect:two}
The parallels between the geometry of isotropic $3$-space and Euclidean $3$-space cannot be understood using the approach involving inner products, since the inner product of isotropic $3$-space is degenerate.
Thus to see the parallelism between the two spaces, we need to understand both geometries without any dependence on metrics.
We propose \emph{Laguerre geometry} \cite{laguerre_sur_1881} as the commonground for understanding the geometry of both Euclidean and isotropic $3$-spaces without any use of inner products (see also \cite[Example~3.3]{pember_discrete_2022}).

In this section, we first give a detailed Laguerre geometric description of basic geometric objects such as planes, spheres, and normals in Euclidean $3$-space without dependence on the standard Euclidean metric, where Euclidean geometric intuition will serve as justifications for the definitions.
We then use the intuition gained from the Euclidean case to define the analogous geometric objects in isotropic $3$-space using Laguerre geometry, making connections to previously known descriptions of isotropic $3$-space.

\subsection{Laguerre geometry of Euclidean \texorpdfstring{$3$}{3}-space}\label{subsect:twoone}
First, we give an account of the Laguerre geometric description of Euclidean $3$-space (see also \cite[\S 3.4]{cecil_lie_2008} or \cite[4.\ Kapitel]{blaschke_vorlesungen_1929}).
Let $\mathbb{L}^4 := \{ (x_0, x_1, x_2, x_3)^t : x_0, x_1, x_2, x_3 \in \mathbb{R}\} $ denote the Minkowski $4$-space with the given metric
	\[
		\langle x, y \rangle = \langle (x_0, x_1, x_2, x_3)^t, (y_0, y_1, y_2, y_3)^t \rangle := -x_0 y_0 + x_1 y_1 + x_2 y_2 + x_3 y_3.
	\]
Choosing a unit timelike vector $\mathfrak{p}$, i.e., $\langle \mathfrak{p}, \mathfrak{p} \rangle = -1$, we have
	\[
		\mathbb{L}^4 = \langle \mathfrak{p} \rangle^\perp \oplus \langle \mathfrak{p} \rangle \cong \mathbb{E}^3 \oplus \langle \mathfrak{p} \rangle
	\]
where $\mathbb{E}^3$ is the standard Euclidean $3$-space.
We note here that any vector $x \in \mathbb{L}^4$ can be expressed as
	\[
		x = x_E + \alpha \mathfrak{p}
	\]
for some $x_E \in \mathbb{E}^3$ and $\alpha \in \mathbb{R}$, and let $\pi_E: \mathbb{L}^4 \to \mathbb{E}^3$ be the orthoprojection, that is, $\pi_E x = x_E$.

Explicitly, one can normalize $\mathfrak{p} = (1,0,0,0)^t$ with a coordinate chart $\psi : \{(\mathbf{x}, \mathbf{y}, \mathbf{z})\} \to \mathbb{E}^3$ so that
	\[
		\mathbb{E}^3 := \{(0, \mathbf{x}, \mathbf{y}, \mathbf{z})^t \in \mathbb{L}^4\}.
	\]

\subsubsection{Spheres and planes}
Let $s \in \mathbb{L}^4$ be any given point, and consider the affine light cone $\mathcal{L}_s$ centered at $s$, that is,
	\[
		\mathcal{L}_s = \{ x \in \mathbb{L}^4 : \langle x - s, x - s \rangle = 0 \}.
	\]
An \emph{oriented sphere} $S$ of Euclidean $3$-space corresponding to $s$ is given as the intersection between $\mathcal{L}_s$ and $\mathbb{E}^3$, i.e., 
	\[
		S = \mathcal{L}_s \cap \mathbb{E}^3 = \{x \in \mathcal{L}_s : \langle x, \mathfrak{p} \rangle = 0\}.
	\]
In fact, writing
	\begin{equation}\label{eqn:sphereOrtho}
		s = c + r \mathfrak{p},
	\end{equation}
for some $c \in \mathbb{E}^3$ and $r \in \mathbb{R}$, we have that the center and the signed radius\footnote{The signature of the radius is interpreted as orientation.} of $S$ are $c$ and $r$, respectively.
This is called the \emph{isotropy projection} (see Figure~\ref{fig:sphere}).
Explicitly, for $s = c + r \mathfrak{p} = (r, c_1, c_2, c_3)^t \in \mathbb{L}^4$,  we have $x \in \mathcal{L}_s \cap \mathbb{E}^3$ if and only if using coordinates,
	\[
		0 = \langle x - s, x - s \rangle = -r^2 + (\mathbf{x} - c_1)^2 + (\mathbf{y} - c_2)^2 + (\mathbf{z} - c_3)^2.
	\]
Thus, every point $s \in \mathbb{L}^4$ corresponds to an oriented sphere in $\mathbb{E}^3 \cong \langle \mathfrak{p} \rangle^\perp$ under the isotropy projection.

In particular, every point $s \in \mathbb{L}^4$ such that $\langle s, \mathfrak{p} \rangle= 0$ corresponds to points (zero radius spheres) in $\mathbb{E}^3$; for this reason the vector $\mathfrak{p}$ is referred to as the \emph{point sphere complex}.

On the other hand, consider an affine isotropic hyperplane $I_{g,q}$ given by a null vector $g \in \mathbb{L}^4$ and some constant $q \in \mathbb{R}$ via
	\begin{equation}\label{eqn:plane}
		I_{g,q} = \{x \in \mathbb{L}^4 : \langle x, g \rangle = q\}.
	\end{equation}
Without loss of generality, we normalize $g$ so that
	\begin{equation}\label{eqn:norm}
		\langle g, \mathfrak{p} \rangle = -1.
	\end{equation}
An \emph{oriented plane} $P$ in $\mathbb{E}^3$ is then given by the intersection between such affine isotropic hyperplanes $I_{g,q}$ and $\mathbb{E}^3 \cong \langle \mathfrak{p} \rangle^\perp$ (see Figure~\ref{fig:plane}).

Explicitly, let $g = (1, g_1, g_2, g_3)^t$.
Then we have $x \in P = I_{g,q} \cap \mathbb{E}^3$ if and only if
	\begin{equation}\label{eqn:planeEuc}
		q = \mathbf{x} g_1 + \mathbf{y} g_2 + \mathbf{z} g_3,
	\end{equation}
so that $P$ is a plane in the coordinate space $\{(\mathbf{x}, \mathbf{y}, \mathbf{z})\}$.

\begin{figure}
	\begin{subfigure}[b]{0.495\textwidth}
		\centering
		\includegraphics[width=\textwidth]{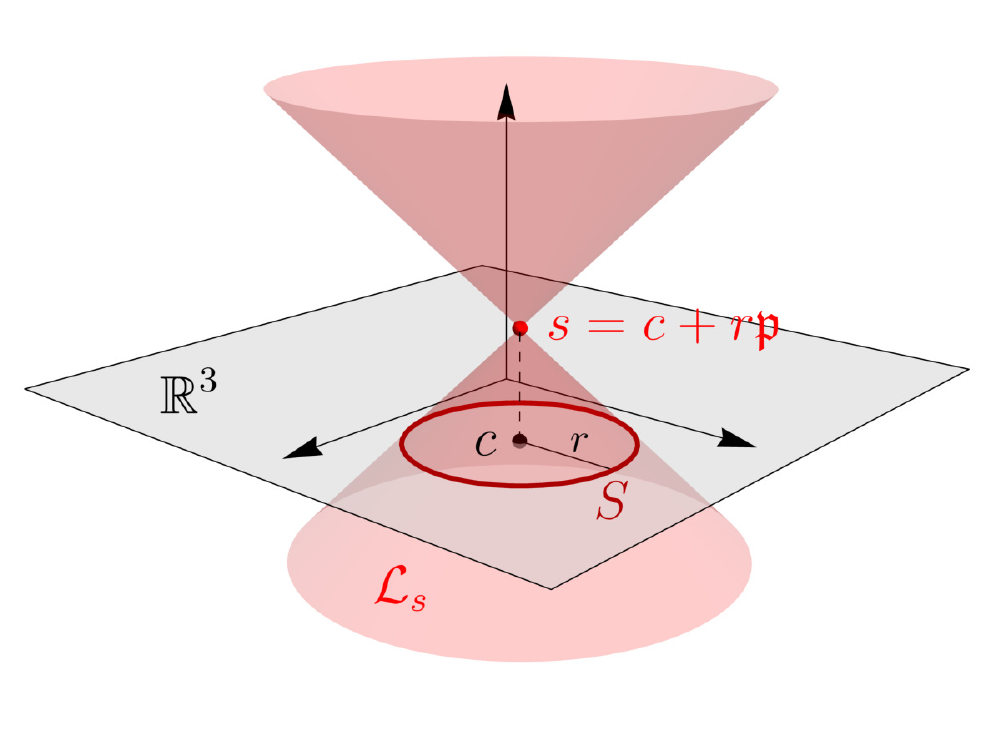}
		\caption{Spheres and isotropy projection}
		\label{fig:sphere}
	\end{subfigure}
	\hfill
	\begin{subfigure}[b]{0.495\textwidth}
		\centering
		\includegraphics[width=\textwidth]{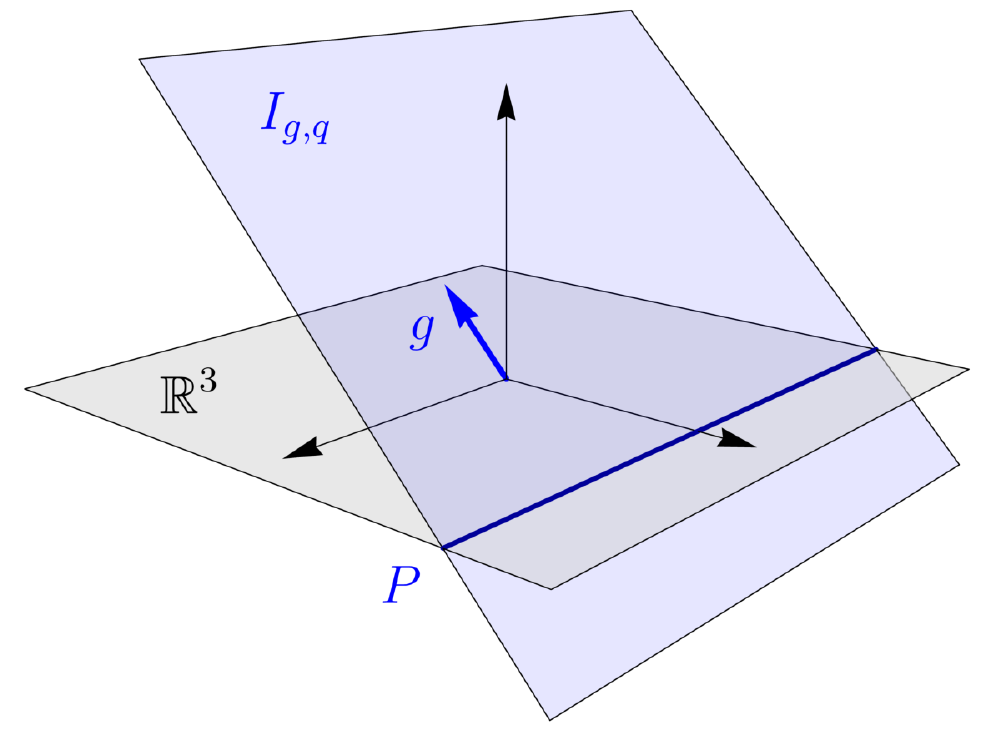}
		\caption{Planes via isotropic hyperplane}
		\label{fig:plane}
	\end{subfigure}
	\caption{Spheres and planes of Euclidean space in Laguerre geometry.}
	\label{fig:sphereplane}
\end{figure}

\begin{remark}
	The prescription of orientation on spheres and planes is essential.
	In the case of spheres, both $s = c + r \mathfrak{p}$ and $\tilde{s} = c - r \mathfrak{p}$ give the same sphere in $\mathbb{E}^3$, but with opposite orientations.
	As we will see similarly for the case of planes, both $I_{g,q}$ and $I_{\tilde{g}, -q}$ for $\tilde{g} = (1, -g_1, -g_2, -g_3)^t$ define two distinct isotropic hyperplanes resulting in the same plane in Euclidean $3$-space, but again with opposite orientations.
	See Figure~\ref{fig:orient}.
	From now on, we \emph{a priori} assume that spheres and planes are prescribed with orientations, and refer to oriented spheres and oriented planes simply as spheres and planes.
\end{remark}

\begin{figure}
	\begin{minipage}{0.4\textwidth}
		\centering
		\includegraphics[width=\textwidth]{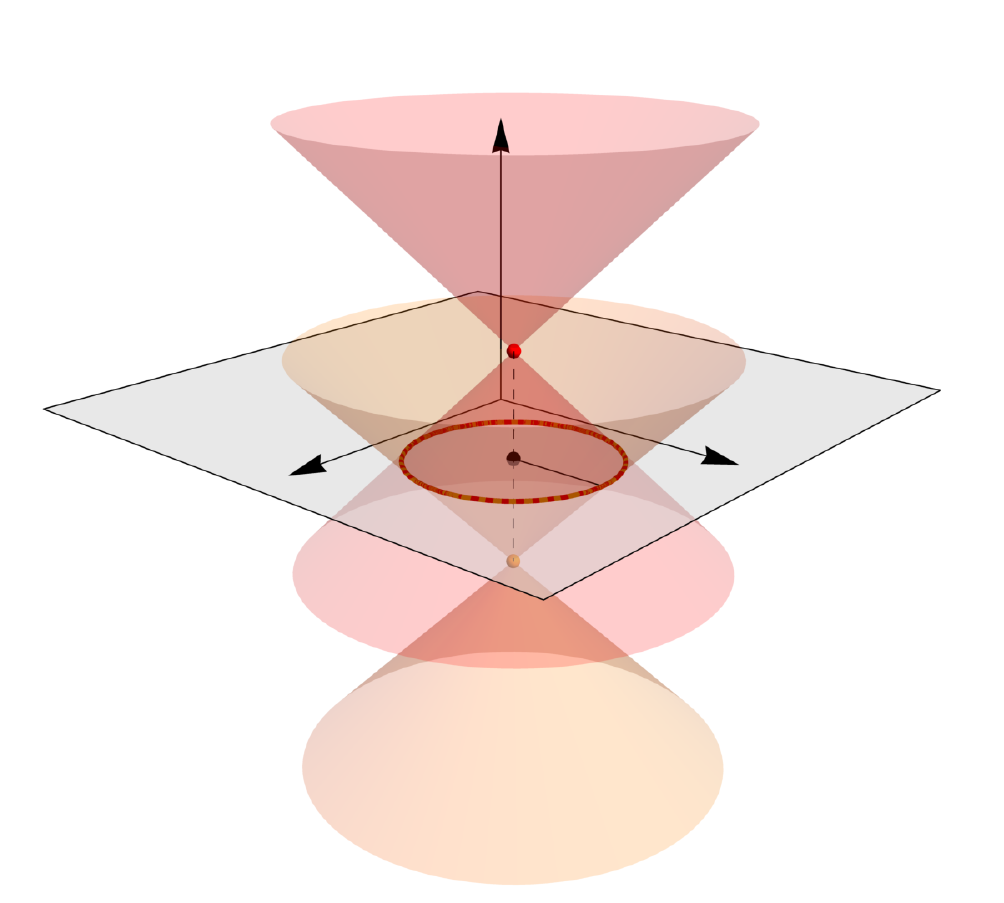}
	\end{minipage}
	\begin{minipage}{0.4\textwidth}
		\centering
		\includegraphics[width=\textwidth]{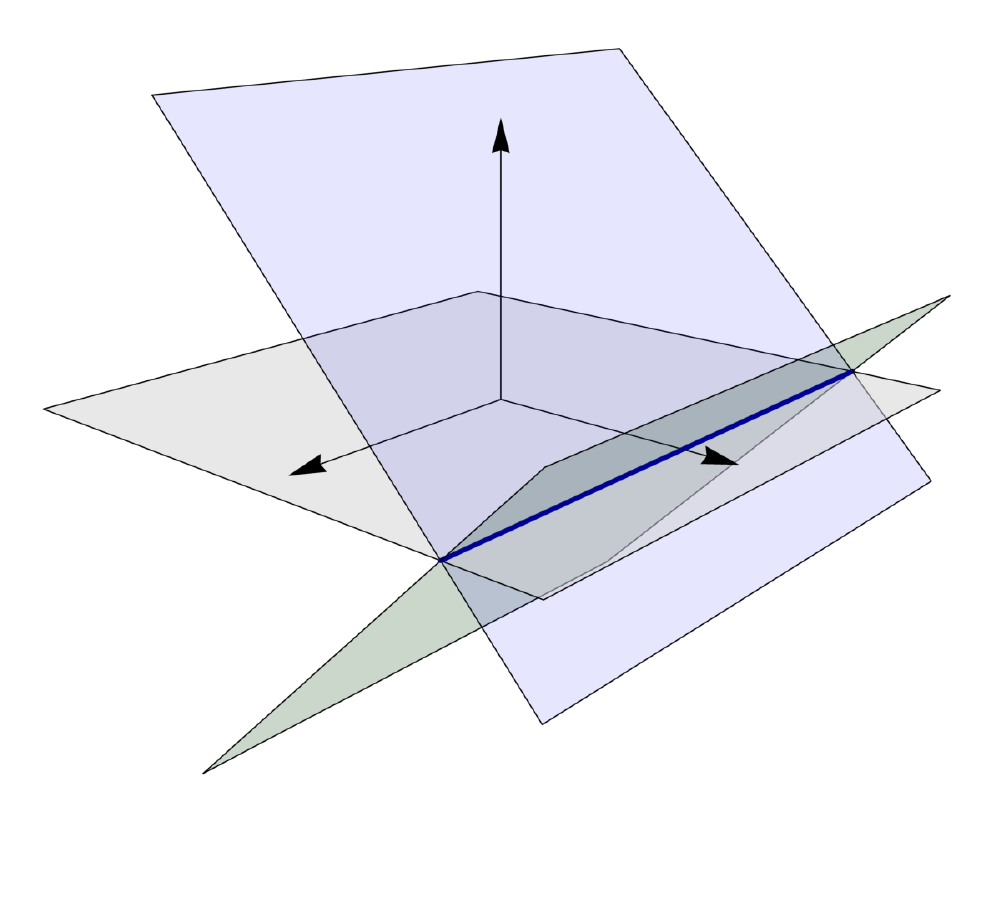}
	\end{minipage}
	\caption{Spheres and planes with different orientations.}
	\label{fig:orient}
\end{figure}

\subsubsection{Contact elements and normals}\label{sect:contactE3}
A sphere is in \emph{oriented contact} with a plane if and only if the affine lightcone $\mathcal{L}_s$ defining the sphere is tangent to the isotropic hyperplane $I_{g,q}$ defining the plane.
Thus a sphere $S$ defined by $s \in \mathbb{L}^4$ and plane $P$ defined by $I_{g,q}$ are in oriented contact if and only if $s \in I_{g,q}$ (see Figure~\ref{fig:contact1}).

Now suppose that $s \in I_{g,q}$, and let $L$ be the affine line along which $I_{g,q}$ is tangent to $\mathcal{L}_s$, i.e.\ $L = I_{g,q} \cap \mathcal{L}_s$.
For any $\tilde{s} \in L$, we have that
	\[
		\langle \tilde{s}, g \rangle = q \quad\text{and}\quad \langle \tilde{s} - s, \tilde{s} - s \rangle = 0.
	\]
Since $s \in I_{g,q}$ implies that $\langle s, g \rangle = q$, we note that
	\[
		\langle \tilde{s} - s, g \rangle = 0,
	\]
implying that $\tilde{s} - s$ and $g$ are both null vectors that are orthogonal to each other.
Thus, $\tilde{s} - s = a g$ for some $a \in \mathbb{R}$, that is, $L$ is an affine null line in the direction of $g$.

Defining $x = L \cap \mathbb{E}^3$, we see that $S$ is in oriented contact with $P$ at the point $x$ (see Figure~\ref{fig:contact2}).
Furthermore, any other $\tilde{s} \in L$ also defines a sphere $\tilde{S}$ in oriented contact with $P$ at the point $x$ (see Figure~\ref{fig:contactElem}).
Thus, an affine null line $L$ represents a $1$-parameter family of spheres in oriented contact with each other, also called a \emph{parabolic sphere pencil} (cf.\ \cite[Definition 1.2.3]{hertrich-jeromin_introduction_2003}); since parabolic sphere pencils uniquely determine the oriented contact at the corresponding points, they are referred to as \emph{contact elements}, and we denote the set of all contact elements by
	\[
		\mathcal{Z}  := \{ \text{affine null lines in $\mathbb{L}^4$} \}.
	\]
	
\begin{figure}
	\begin{subfigure}[b]{0.495\textwidth}
		\centering
		\includegraphics[width=\textwidth]{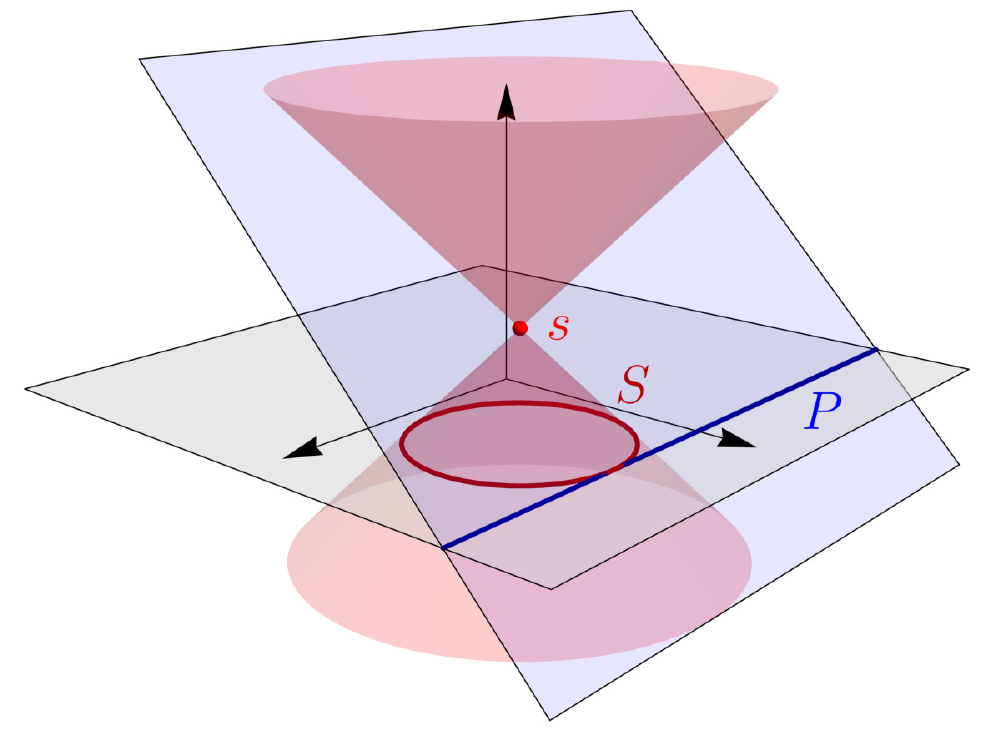}
		\caption{Spheres and planes in oriented contact}
		\label{fig:contact1}
	\end{subfigure}
	\begin{subfigure}[b]{0.495\textwidth}
		\centering
		\includegraphics[width=\textwidth]{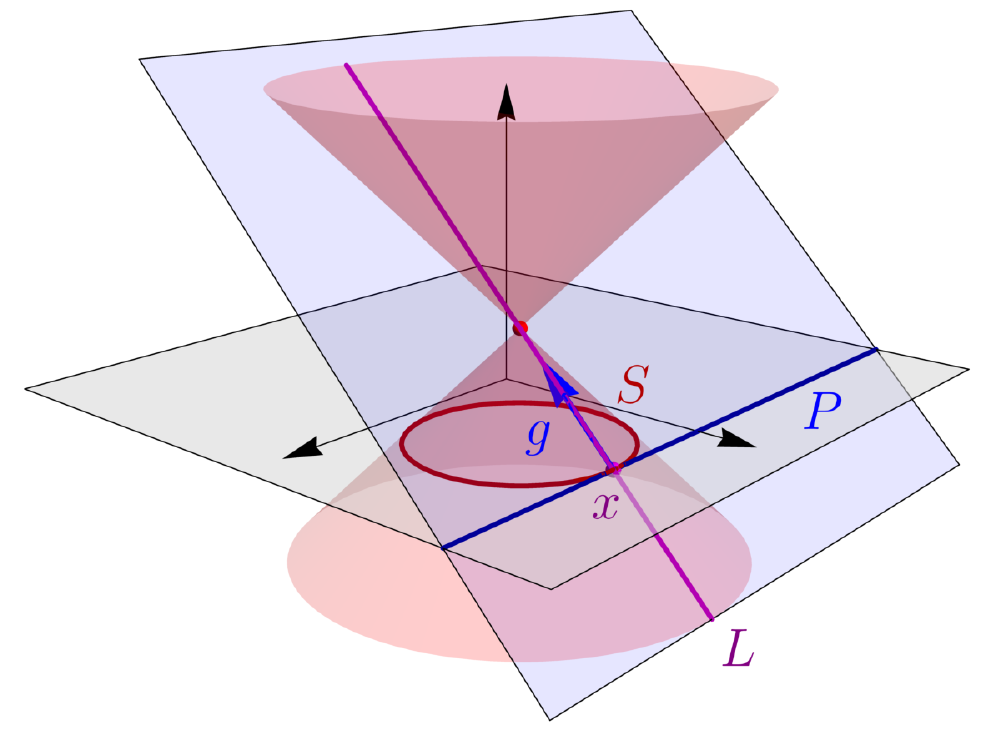}
		\caption{Point of contact via affine null line}
		\label{fig:contact2}
	\end{subfigure}
	\begin{subfigure}[b]{0.495\textwidth}
		\centering
		\includegraphics[width=\textwidth]{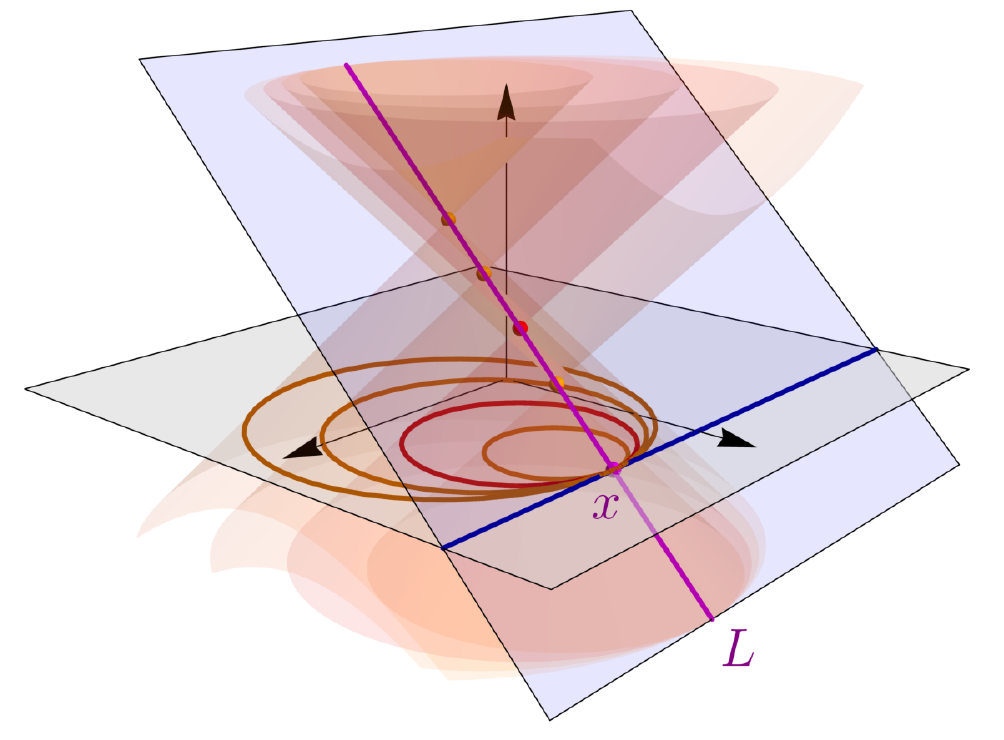}
		\caption{Contact element of a plane at a point}	
		\label{fig:contactElem}
	\end{subfigure}
	\begin{subfigure}[b]{0.495\textwidth}
		\centering
		\includegraphics[width=\textwidth]{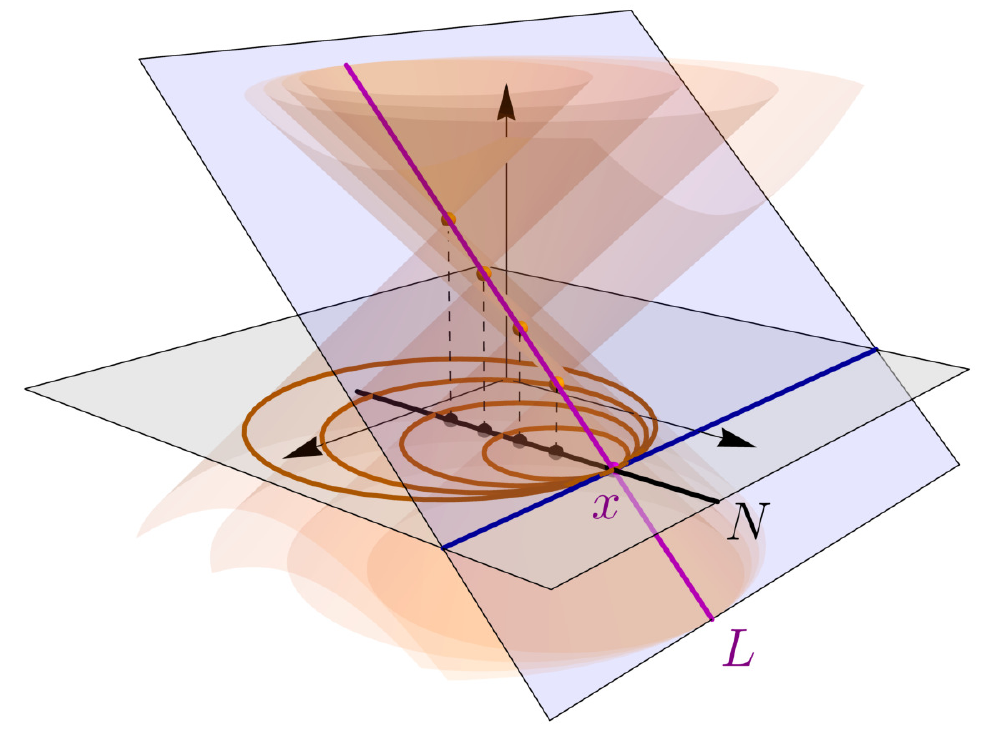}
		\caption{Normal line via orthoprojection}
		\label{fig:planeNormal}
	\end{subfigure}
	\caption{Contact elements of Euclidean space in Laguerre geometry.}
	\label{fig:contact}
\end{figure}

To recover the normal to the plane $P$ within the Euclidean space in this setup, let $L$ be a contact element of $P$ at $x$.
Consider $N := \pi_E L$, the orthoprojection of $L$ to $\mathbb{E}^3$.
Then the centers of the spheres represented by $L$ lie on $N$ \eqref{eqn:sphereOrtho}; thus, basic Euclidean geometry tells us that $N$ is orthogonal to $P$ (see Figure~\ref{fig:planeNormal}).
Intersecting $N$ with the unit sphere centered at $x$, we obtain the unit normal $\nu$ to the plane $P$ in Euclidean $3$-space.\footnote{The intersection yields two candidates for the unit normal; however, the orientation on the spheres uniquely determines the unit normal.}

In fact, we can explicitly calculate the unit normal $\nu$ to the plane from the isotropic hyperplane $I_{g,q}$ generating the plane.
Recall that $L$ is an affine null line through $x$ pointing in the direction of $g = (1, g_1, g_2, g_3)^t$.
Thus $N = \pi_E L$ is a line through $x$ pointing in the direction of $\pi_E g = g_E = (0, g_1, g_2, g_3)^t$, which is the unit Euclidean normal to the plane \eqref{eqn:planeEuc} where the unit length follows from the normalization of $g$ \eqref{eqn:norm}.

Conversely, given a unit normal $\nu$ to a plane $P$ in Euclidean space, one can easily recover $g$ via $g = \nu + \mathfrak{p}$.

\subsubsection{Surfaces and their contact elements}
Let $x: \Sigma \to \mathbb{E}^3$ be an immersion defined on some simply-connected domain $\Sigma$.
For every point $p \in \Sigma$, the tangent plane to $x$ at $x(p)$, denoted by $T_{x(p)} x$, gives rise to a contact element of $T_{x(p)} x$ at $x(p)$, represented by an affine null line $L(p) \in \mathcal{Z}$.
The contact element $L(p)$ also gives rise to the unit normal $\nu(p)$ to $T_{x(p)} x$ at $x(p)$, so that $\nu: \Sigma \to S^2 \subset \mathbb{E}^3$ is the Gauss map of $x$.
Thus a surface, together with its contact elements, can be thought of as a map into the set of contact elements $L : \Sigma \to \mathcal{Z}$, which we call the \emph{contact lift} of $x$.
In fact, we have
	\begin{equation}\label{eqn:contactLift}
		L = x + \langle \nu + \mathfrak{p}\rangle.
	\end{equation}
	
Now, we aim to understand when a surface can be recovered from a map into the set of contact elements $L : \Sigma \to \mathcal{Z}$, where we view $L$ as a null line bundle.
For this, we first define the following notions:
\begin{definition}\label{def:Legendre}
	Let $L : \Sigma \to \mathcal{Z}$ be a map into the set of contact elements, and let $G : \Sigma \to \mathbb{P}(\mathcal{L})$ such that $G(p) \parallel L(p)$ for every $p \in \Sigma$.
	$L$ is called a \emph{Legendre map} or a \emph{frontal} if $L$ satisfies the contact condition on $\Sigma$:
		\begin{equation}\label{eqn:contact}
			\langle\dif{\ell}, g \rangle = 0\text{ for any sections $\ell \in \Gamma L$ and $g \in \Gamma G$.}
		\end{equation}
	A Legendre map $L$ further satisfying the immersion condition on $\Sigma$:
		\begin{equation}\label{eqn:immersion}
			|{\dif{\ell}}|^2(X) := \langle \dif{\ell}(X),\dif{\ell}(X) \rangle = 0\text{ for any $\ell \in \Gamma L$ implies $X = 0$}
		\end{equation}
	is called a \emph{Legendre immersion}, or a $\emph{front}$.
\end{definition}

From a Legendre map $L$, one can recover the surface by taking the intersection with $\mathbb{E}^3$, i.e.\ $x = L \cap \langle \mathfrak{p} \rangle^\perp : \Sigma \to \mathbb{E}^3$.
Under the assumption that $x$ so recovered immerses, we can recover the Gauss map as follows: let $G : \Sigma \to \mathbb{P}(\mathcal{L})$ be as in Definition~\ref{def:Legendre}, and consider the affine hyperplane $\Pi$ parallel to $\mathbb{E}^3$ defined via 
	\[
		\Pi := \{x \in \mathbb{L}^4 : \langle x, \mathfrak{p} \rangle = -1\}.
	\]
Let $g: \Sigma \to \mathcal{L}$ be defined via $g := G \cap \Pi$, that is,\ $g \in \Gamma G$ with $\langle g, \mathfrak{p}\rangle = -1$.
Then we have
	\[
		g = \nu + \mathfrak{p}
	\]
for some $\nu: \Sigma \to S^2 \subset \mathbb{E}^3$.
Therefore, $\pi_E g = \nu$ is the $\emph{Gauss map}$ of $x$, and we call $g$ the \emph{lightlike Gauss map} \cite[p.\ 517]{izumiya_lightcone_2004}.

To justify our Definition~\ref{def:Legendre} of frontals and fronts, recall that $x: \Sigma \to \mathbb{R}^3$ is called a frontal if one can find a unit normal vector field $\nu: \Sigma \to S^2$, that is, $\dif{x} \perp \nu$, while $x$ is a front if $(x,\nu) : \Sigma \to \mathbb{R}^3 \times S^2$ is an immersion.
Now, let $L: \Sigma \to \mathcal{Z}$ with $x := L \cap \mathbb{R}^3$, and define $G : \Sigma \to \mathbb{P}(\mathcal{L})$ as in Definition~\ref{def:Legendre} so that $L = x + G$.
Again, let $g \in \Gamma G$ such that $\langle g, \mathfrak{p} \rangle = -1$, and define $\nu := \pi_E g : \Sigma \to S^2$, i.e.\ $g = \nu + \mathfrak{p}$.
	
To consider the contact condition \eqref{eqn:contact}, first note that the choice of section $g \in \Gamma G$ does not matter; thus take $g = \nu + \mathfrak{p}$, and let $\ell \in \Gamma L$ be any section.
Then there is a function $\alpha : \Sigma \to \mathbb{R}$ satisfying
	\[
		\ell = x + \alpha g,
	\]
so that
	\[
		\dif{\ell} = \dif{x} + \dif{\alpha} \, g+ \alpha\dif{\nu}.
	\]
Therefore,
	\[
		\langle \dif{\ell}, g \rangle = \langle \dif{x} + \dif{\alpha}\, g + \alpha\dif{\nu}, g \rangle = \langle \dif{x}, \nu \rangle.
	\]
Hence we have that $L$ is Legendre if and only if $\nu$ is a unit normal vector field of $x$, i.e.\ $x$ is a frontal.
	
For the immersion condition \eqref{eqn:immersion}, assume that $L$ is Legendre, and suppose  for any section $\ell \in \Gamma L$ that
	\[
		|{\dif{\ell}}|^2(X) = \langle \dif{\ell}(X),\dif{\ell}(X) \rangle = 0.
	\]
Equivalently, we have for any $\alpha: \Sigma \to \mathbb{R}$,
	\begin{equation}\label{eqn:almostcurvature}
		\begin{aligned}
			0 &= \langle  \dif{x}(X) + \dif{\alpha}(X) \, g+ \alpha\dif{\nu}(X),  \dif{x}(X) + \dif{\alpha} \, g+ \alpha\dif{\nu}(X)\rangle \\
				&= |{\dif{x}}|^2(X) + 2\alpha\langle\dif{x}(X),\dif{\nu}(X)\rangle + \alpha^2 |{\dif{\nu}}|^2(X) \\
				&= \langle \dif{x}(X) + \alpha \dif{\nu}(X), \dif{x}(X) + \alpha \dif{\nu}(X) \rangle.
		\end{aligned}
	\end{equation}
Thus $L$ satisfies the immersion condition \eqref{eqn:immersion} if and only if $x$ is a front.

\begin{remark}
	We remark here that \eqref{eqn:almostcurvature} holds true if $X$ is a principal direction and $\alpha$ is the radius of the corresponding curvature sphere.
	Noting that any section $\ell \in \Gamma L$ represents a sphere congruence enveloping $x$, we see that $\ell$ is a curvature sphere congruence if and only if
		\[
			|{\dif{\ell}}|^2(X) = 0
		\]
	for some non-zero $X \in \Gamma T\Sigma$.
	Moreover, such $X$ gives the principal direction.
\end{remark}

\subsubsection{Isometries fixing the origin} Finally, denoting by $O(3,1)$ the orthogonal group of $\mathbb{L}^4$, the isometries of $\mathbb{E}^3$ fixing the origin are given by all $A \in O(3,1)$ such that $A\mathfrak{p} = \mathfrak{p}$.

\subsubsection{Summary}\label{sec:Euclid}
	We conclude this section by providing a summary of the Laguerre representatives of Euclidean objects:
	\begin{itemize}
		\item The space form $\mathbb{E}^3$ is obtained via a timelike point sphere complex $\mathfrak{p}$ through the straightforward isomorphism $\mathbb{E}^3 \cong \langle \mathfrak{p} \rangle^\perp$.
		\item An oriented sphere with center $c \in \mathbb{E}^3$ and radius $r \in \mathbb{R}$ is represented by $s = c + r \mathfrak{p} \in \mathbb{L}^4$, and recovered by taking the intersection of $\mathbb{E}^3$ with the affine lightcone centered at $s$.
		\item An oriented plane $P$ is represented by an isotropic hyperplane $I_{g,q}$ determined by a null vector $g$ and a constant $q$.
		\item A contact element (elliptic sphere pencil) is represented by an affine null line $L$.
		\item The normal line $N$ to the plane $P$ is given by the orthoprojection of an affine null line $L \subset I_{g,q}$.
		\item A surface together with its tangent bundle is given by the map into the set of contact elements $\mathcal{Z}$.
	\end{itemize}

\subsection{Laguerre geometry of isotropic \texorpdfstring{$3$}{3}-space}\label{subsect:twotwo}
Now we turn our attention to isotropic $3$-space and define its basic geometric objects analogously to those of the Euclidean case using Laguerre geometry summarized in Subsection~\ref{sec:Euclid}.

Choosing a point sphere complex $\mathfrak{p}$ so that $\langle \mathfrak{p}, \mathfrak{p} \rangle = 0$, the isotropic $3$-space $\mathbb{I}^3$ can be defined via
	\[
		\mathbb{I}^3 := \langle \mathfrak{p} \rangle^\perp := \{ x \in \mathbb{L}^4 : \langle x, \mathfrak{p} \rangle = 0\}.
	\]
Now we let $\tilde{\mathfrak{p}} \in \mathcal{L}$ with $\langle \mathfrak{p}, \tilde{\mathfrak{p}} \rangle = 1$, so that
	\[
		\mathbb{L}^4 = \mathbb{I}^3 \oplus \langle \tilde{\mathfrak{p}} \rangle.
	\]
Then any $x \in \mathbb{L}^4$ can now be written as
	\[
		x = x_I + \alpha \tilde{\mathfrak{p}}
	\]
for $x_I \in \mathbb{I}^3$ and $\alpha \in \mathbb{R}$.
We define the projection $\pi_I: \mathbb{L}^4 \to \mathbb{I}^3$ via $\pi_I x = x_I$.

To make explicit connection with previously known definitions of planes and spheres of isotropic $3$-space (see, for example, \cite{pottmann_discrete_2007}), we normalize $\mathfrak{p}$ and $\tilde{\mathfrak{p}}$ so that
	\[
		\mathfrak{p} = (1,0,0,1)^t \quad\text{and}\quad \tilde{\mathfrak{p}} = \frac{1}{2}(-1,0,0,1)^t.
	\]
Let $\psi : \{(\mathbf{l}, \mathbf{x}, \mathbf{y})\} \to \mathbb{I}^3$ be a coordinate chart so that
	\[
		\mathbb{I}^3 = \{ (\mathbf{l}, \mathbf{x}, \mathbf{y}, \mathbf{l})^t \in \mathbb{L}^4\}.
	\]
Then the metric $g_{\mathbb{I}^3}$ of $\mathbb{I}^3$ is endowed from the ambient space $\mathbb{L}^4$ so that
	\[
		g_{\mathbb{I}^3} = \dif{\mathbf{x}}^2 + \dif{\mathbf{y}}^2.
	\]
\begin{remark}
	When describing or visualizing isotropic $3$-space using the coordinate space $\{(\mathbf{l}, \mathbf{x}, \mathbf{y})\}$, we will refer to the $\mathbf{l}$--direction as vertical.
\end{remark}

\begin{remark}
	We will refer to both $(\mathbf{l}, \mathbf{x}, \mathbf{y}, \mathbf{l})^t \in \mathbb{L}^4$ and $\psi^{-1}((\mathbf{l}, \mathbf{x}, \mathbf{y}, \mathbf{l})^t) = (\mathbf{l}, \mathbf{x}, \mathbf{y})$ as being in the isotropic $3$-space.
	
\end{remark}
	
\subsubsection{Spheres in isotropic space} Analogous to the Euclidean case, let an oriented sphere $S$ in isotropic space be defined via $S := \mathcal{L}_s \cap \mathbb{I}^3$, where $\mathcal{L}_s$ denotes an affine lightcone centered at any $s \in \mathbb{L}^4$.
To find the ``center'' and the ``radius'' of the sphere in isotropic space, let $s = (s_0, s_1, s_2, s_3)^t$, and note that $x = (x_0, x_1, x_2, x_3)^t \in \mathcal{L}_s$ if and only if
	\[
		(x_1 - s_1)^2 + (x_2 - s_2)^2 = ((x_0 - s_0) + (x_3 - s_3))((x_0 - s_0) - (x_3 - s_3)) > 0.
	\]
Thus, defining
	\[
		\alpha := (x_0 - s_0) + (x_3 - s_3) \quad\text{and}\quad \beta := (x_0 - s_0) - (x_3 - s_3)
	\]
we can write
	\[
		x_1 - s_1 = \sqrt{\alpha \beta} \cos \theta \quad\text{and}\quad x_2 - s_2 = \sqrt{\alpha \beta} \sin \theta
	\]
for some $\theta \in \mathbb{R}$.

For $c = (c_0, c_1, c_2, c_0)^t := \pi_I s$ and some $r  \in \mathbb{R}$, we can write
	\begin{equation}\label{eqn:sphereI3up}
		s = c + r \tilde{\mathfrak{p}},
	\end{equation}
so that $\mathcal{L}_s$ can be parametrized by $(\alpha, \beta, \theta)$ via
	\begin{align*}
		\mathcal{L}_s &= (\tfrac{1}{2}(\alpha + \beta) + s_0,  \sqrt{\alpha \beta} \cos \theta + s_1,  \sqrt{\alpha \beta} \sin \theta + s_2, \tfrac{1}{2}(\alpha - \beta) + s_3) \\
			&= (\tfrac{1}{2}(\alpha + \beta) + c_0 - \tfrac{r}{2},  \sqrt{\alpha \beta} \cos \theta + c_1,  \sqrt{\alpha \beta} \sin \theta + c_2, \tfrac{1}{2}(\alpha - \beta) + c_0 + \tfrac{r}{2}).
	\end{align*}
Then for $x \in S = \mathcal{L}_s \cap \mathbb{I}^3$, we need
	\[
		\tfrac{1}{2}(\alpha + \beta) + c_0 - \tfrac{r}{2} = \tfrac{1}{2}(\alpha - \beta) + c_0 + \tfrac{r}{2},
	\]
or equivalently, $\beta = r$.
Thus, we have
	\[
		S = (\tfrac{\alpha}{2} + c_0,  \sqrt{\alpha r} \cos \theta + c_1,  \sqrt{\alpha r} \sin \theta + c_2, \tfrac{\alpha}{2} + c_0),
	\]
or in terms of the coordinates $(\mathbf{l}, \mathbf{x}, \mathbf{y})$,
	\begin{equation}\label{eqn:sphereI3}
		(\mathbf{x} - c_1)^2 + (\mathbf{y} - c_2)^2 = \alpha r = 2r (\mathbf{l}  - c_0).
	\end{equation}
We say that the sphere given by \eqref{eqn:sphereI3up} or \eqref{eqn:sphereI3} has \emph{center} $c = (c_0, c_1, c_2)$ and \emph{radius} $r$.

If $r \neq 0$, then this gives a \emph{sphere of non-zero radius} in isotropic $3$-space, also called an \emph{i-sphere of the parabolic type} in \cite[Equation (3)]{pottmann_discrete_2007}, which is a paraboloid of rotation in the coordinate space.
In particular, if $r = 1$, then we call it a \emph{unit sphere}.
Furthermore, if a unit sphere is centered at the origin, then
	\[
		\mathbf{x}^2 + \mathbf{y}^2 = 2 \mathbf{l},
	\]
which coincides with the \emph{isotropic unit sphere} given in \cite[Equation (5)]{pottmann_discrete_2007}.
It can be checked that the metric of these spheres induced by the ambient space is Riemannian; hence, we refer to these spheres as \emph{spacelike spheres}.

On the other hand, if $r \to 0$ and $c_0 \to -\infty$ at suitable rates so that $r c_0 \to R \in \mathbb{R}_{\geq 0}$, then we also have
	\[
		(\mathbf{x} - c_1)^2 + (\mathbf{y} - c_2)^2 = R.
	\]
These are \emph{spheres of zero radius} in isotropic $3$-space, also called the \emph{i-spheres of the cylindrical type}, and are vertical lines or a circular cylinders in the coordinate space.
Again, one can check that the metric of these spheres induced by the ambient space is degenerate, and we refer to these spheres as \emph{lightlike spheres}.

\begin{figure}
	\centering
	\includegraphics[width=0.48\textwidth]{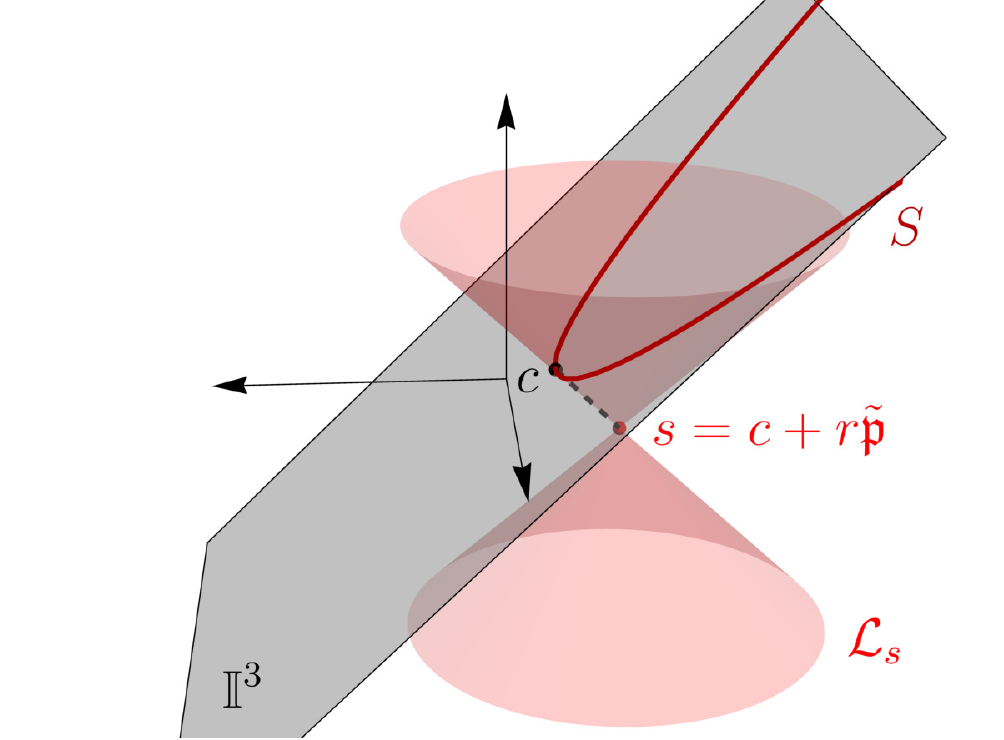}
	\caption{Spheres in isotropic geometry via isotropy projection}
	\label{fig:i3sphere}
\end{figure}

\subsubsection{Planes in isotropic space} Analogous to the Euclidean case, oriented planes $P$ in isotropic $3$-space are given by the intersection of affine isotropic hyperplanes $I_{g,q}$ and $\mathbb{I}^3 \cong \langle \mathfrak{p} \rangle^\perp$.
Explicitly, let an affine isotropic hyperplane $I_{g,q}$ be as in \eqref{eqn:plane} for some lightlike $g = (g_0, g_1, g_2, g_3)^t$.
Then $x \in P = I_{g,q} \cap \mathbb{I}^3$ if and only if 
	\begin{equation}\label{eqn:planeI3}
		q = \mathbf{l} (g_3 - g_0) + \mathbf{x} g_1 + \mathbf{y} g_2
	\end{equation}
in the coordinate space.

If $g_0 = g_3$ so that $g \in \mathbb{I}^3$, then we obtain vertical planes in the coordinate space, which we refer to as \emph{lightlike planes} since the metric induced on the plane is degenerate (see Figure~\ref{fig:i3isoplane}); if $g_0 \neq g_3$ so that $g \not \in \mathbb{I}^3$, then we obtain non-vertical planes in the coordinate space, which we refer to as \emph{spacelike planes} (see Figure~\ref{fig:i3plane}).

Since $\langle g, \mathfrak{p} \rangle = g_3 -g_0$, in the case of spacelike planes, we normalize $g$ without loss of generality so that $\langle g, \mathfrak{p} \rangle = 1$.
Then $g = (g_0, g_1, g_2, g_0 + 1)^t$, and the equation of the plane becomes
	\begin{equation}\label{eqn:planeI3norm}
		q = \mathbf{l} + \mathbf{x} g_1 + \mathbf{y} g_2.
	\end{equation}

\begin{remark}
	Note that one cannot have isotropic hyperplanes in $\mathbb{L}^4$ with $g \in \mathbb{E}^3$; thus, there are no lightlike planes in $\mathbb{E}^3$.
	In fact, when considering Laguerre geometry of Minkowski $3$-space $\mathbb{L}^3$ in a similar fashion (via a unit spacelike point sphere complex), the spacelike planes of $\mathbb{L}^3$ are obtained by the isotropic hyperplanes $I_{g,q}$ defined by $g \not\in \mathbb{L}^3$, while the lightlike planes are obtained by those $g \in \mathbb{L}^3$.
\end{remark}

\begin{figure}
	\begin{subfigure}[b]{0.495\textwidth}
		\centering
		\includegraphics[width=\textwidth]{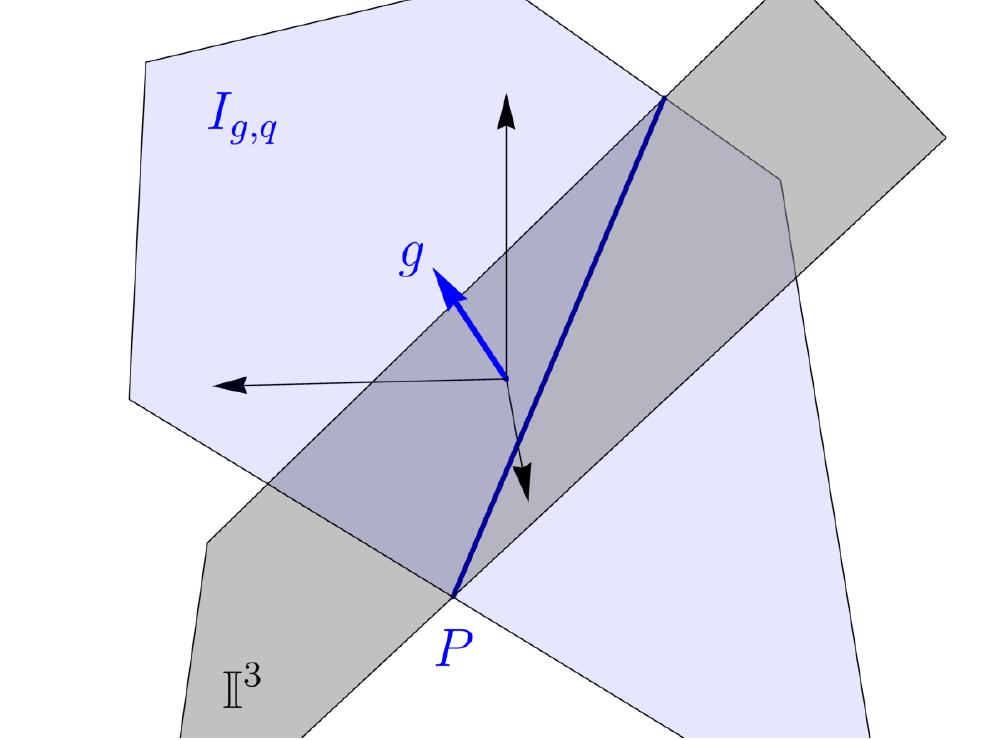}
		\caption{Spacelike planes}
		\label{fig:i3plane}
	\end{subfigure}
	\hfill
	\begin{subfigure}[b]{0.495\textwidth}
		\centering
		\includegraphics[width=\textwidth]{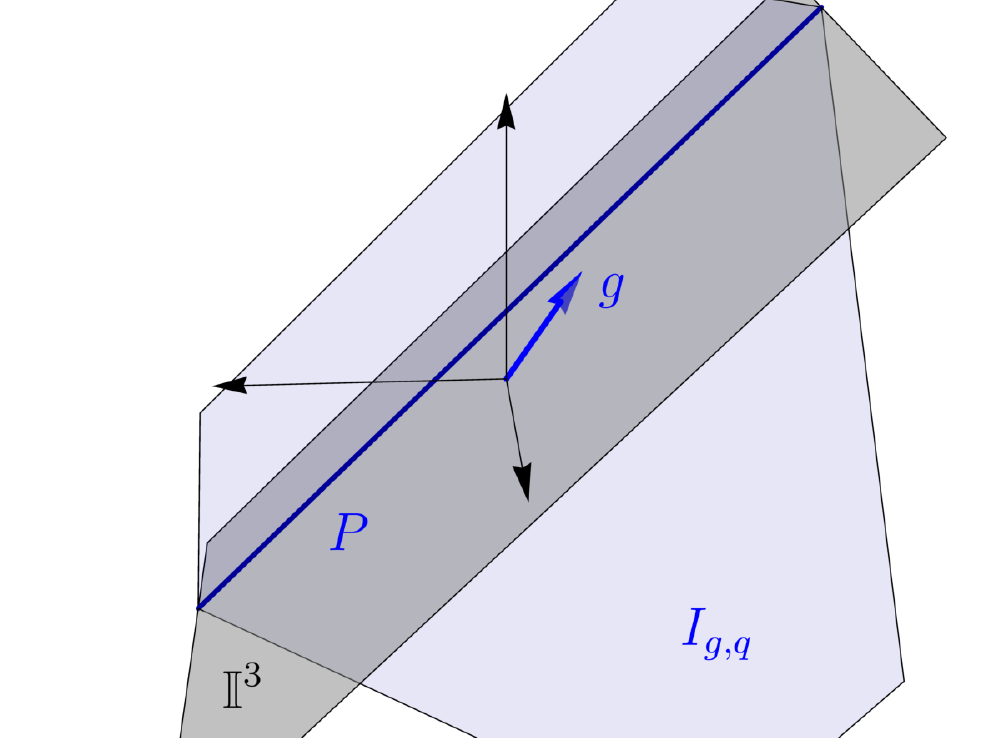}
		\caption{Lightlike planes}
		\label{fig:i3isoplane}
	\end{subfigure}
	\caption{Spacelike planes and lightlike planes of isotropic space  via isotropic hyperplanes in Laguerre geometry.}
	\label{fig:i3planes}
\end{figure}

\subsubsection{Contact elements and normals}
Since we are mainly interested in spacelike surfaces, i.e.\ surfaces whose induced metric is Riemannian, we will only discuss contact elements for spacelike planes.
Let $P$ be a spacelike plane in $\mathbb{I}^3$ given by affine isotropic hyperplane $I_{g,q}$ with lightlike $g = (g_0, g_1, g_2, g_0 + 1)^t$.
As in the Euclidean case, the sphere $S$ given by $s \in \mathbb{L}^4$ is in oriented contact with $P$ if and only if $\mathcal{L}_s$ is tangent to $I_{g,q}$, or equivalently, $s \in I_{g,q}$ (see Figure~\ref{fig:i3contact1}).

Now, if we let $L$ be the affine line along which $I_{g,q}$ and $\mathcal{L}_s$ is tangent as in Section~\ref{sect:contactE3}, then we have again that $L$ is an affine null line pointing in the direction of $g \not\in \mathbb{I}^3$.
Thus, we can uniquely find $x = L \cap \mathbb{I}^3$  (see Figure~\ref{fig:i3contact2}), and see that $L$ represents all spheres that are in oriented contact with the plane $P$ at the point $x$ (see Figure~\ref{fig:i3contactElem}).

Let $N := \pi_I L$ be the projection of $L$, and for $g_I := \pi_I g$, write $g$ as
	\[
		g = g_I + \alpha \tilde{\mathfrak{p}}.
	\]
Then $N$ is a line in the direction of $g_I$, and the centers of the spheres represented by $L$ lie in $N$ \eqref{eqn:sphereI3up}.
We then take the unique intersection $\nu$ of $N$ with the unit sphere centered at $x$: we call $\nu$ the \emph{unit normal} to the plane $P$  (see Figure~\ref{fig:i3planeNormal}).

\begin{figure}
	\begin{subfigure}[b]{0.495\textwidth}
		\centering
		\includegraphics[width=\textwidth]{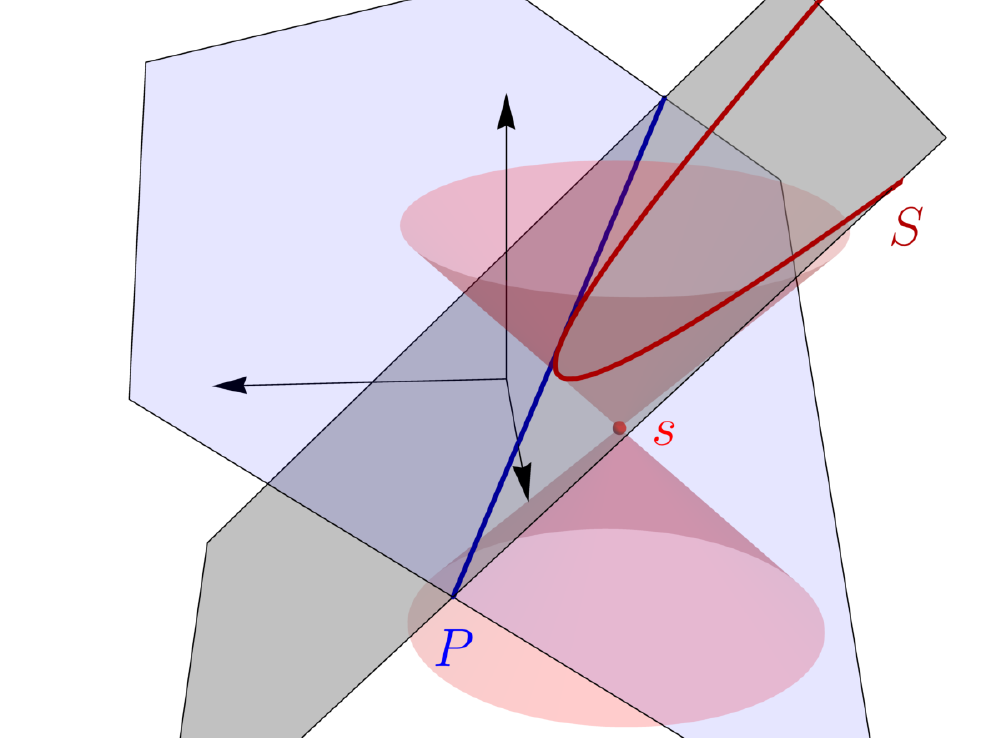}
		\caption{Spheres and planes in oriented contact}
		\label{fig:i3contact1}
	\end{subfigure}
	\begin{subfigure}[b]{0.495\textwidth}
		\centering
		\includegraphics[width=\textwidth]{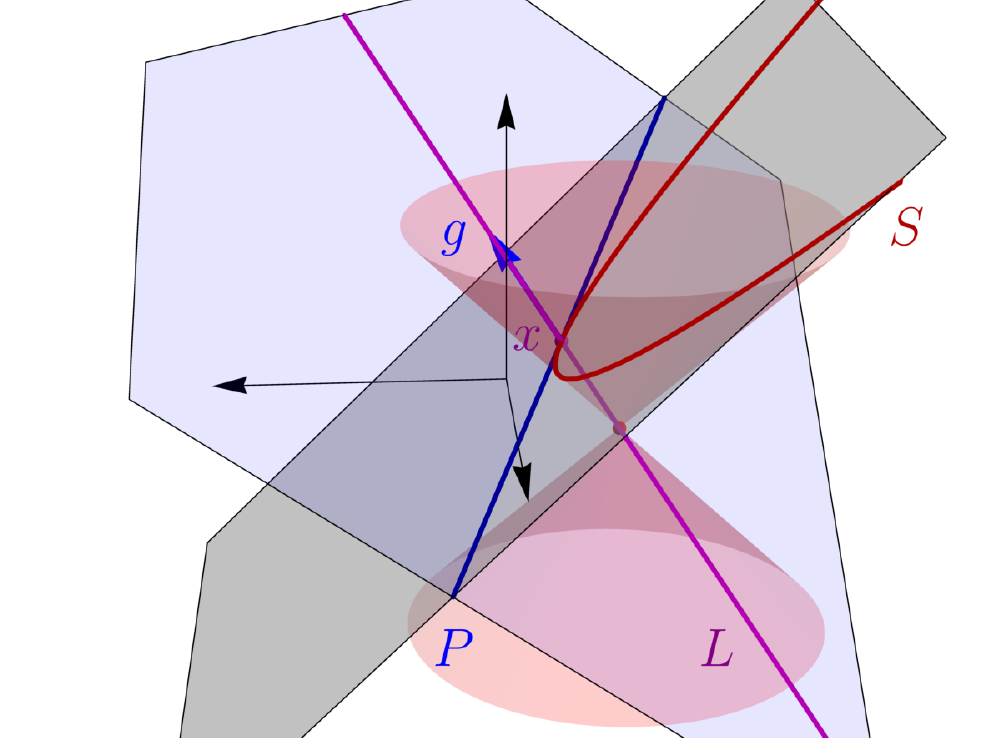}
		\caption{Point of contact via affine null line}
		\label{fig:i3contact2}
	\end{subfigure}
	\begin{subfigure}[b]{0.495\textwidth}
		\centering
		\includegraphics[width=\textwidth]{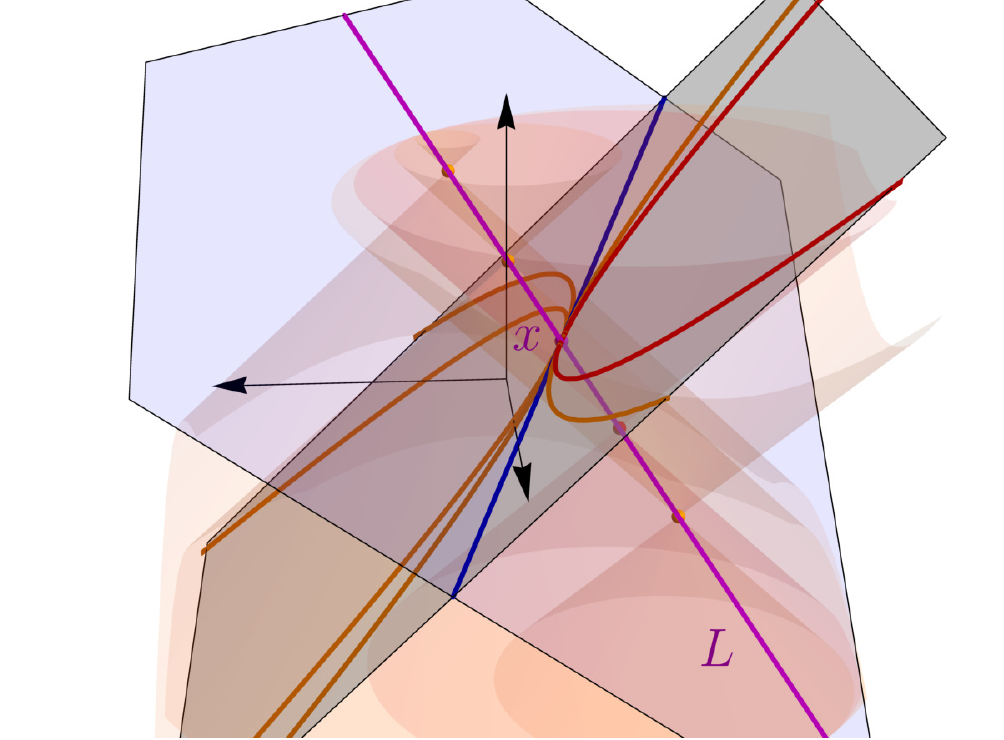}
		\caption{Contact element of a plane at a point}	
		\label{fig:i3contactElem}
	\end{subfigure}
	\begin{subfigure}[b]{0.495\textwidth}
		\centering
		\includegraphics[width=\textwidth]{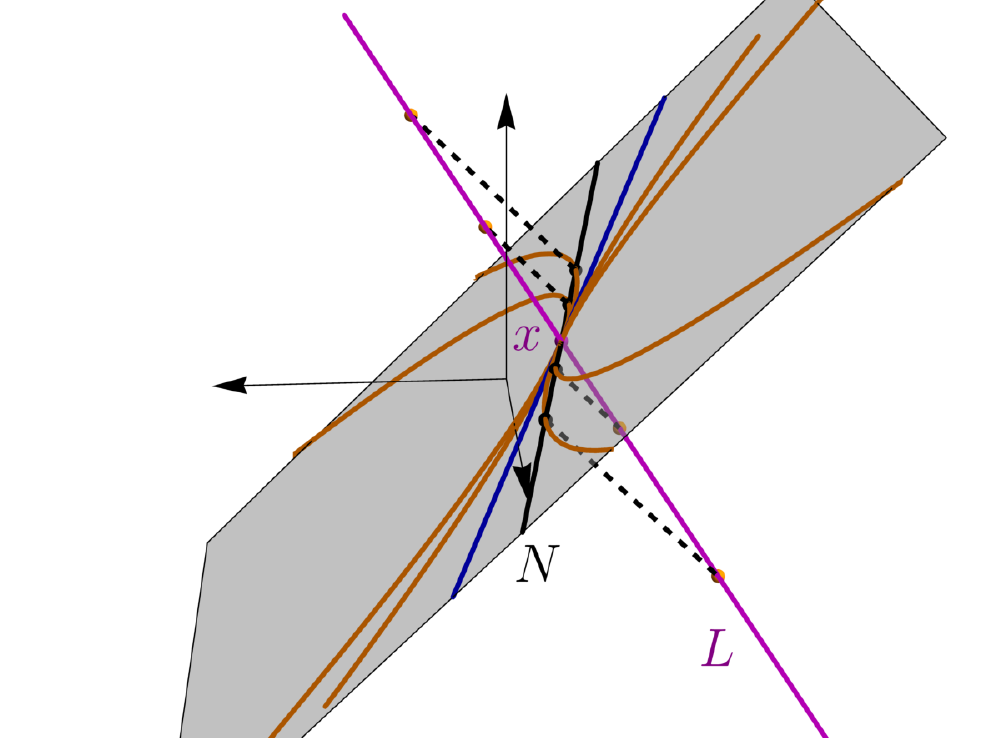}
		\caption{Normal line via projection}
		\label{fig:i3planeNormal}
	\end{subfigure}
	\caption{Contact elements of isotropic space in Laguerre geometry.}
	\label{fig:I3contact}
\end{figure}

To find $\nu$ explicitly, we first see that
	\[
		1 = \langle g, \mathfrak{p} \rangle = \langle g_I, \mathfrak{p} \rangle + \alpha \langle \tilde{\mathfrak{p}}, \mathfrak{p} \rangle
			 = \alpha,
	\] 
since $g_I \in \mathbb{I}^3$.
Hence
	\[
		g_I = g - \tilde{\mathfrak{p}} = (g_0 + \tfrac{1}{2}, g_1, g_2, g_0 + \tfrac{1}{2})^t.
	\]
To see when $N$ intersects the unit sphere centered at $x =: (c_0, c_1, c_2, c_0)^t$, we parametrize $N$ via
	\[
		N(a) = a g_I + x = a (g_0 + \tfrac{1}{2}, g_1, g_2, g_0 + \tfrac{1}{2})^t + (c_0, c_1, c_2, c_0)^t,
	\]
so that
	\[
		\psi^{-1}(N(a) )= a (g_0 + \tfrac{1}{2}, g_1, g_2) + (c_0, c_1, c_2).
	\]
Now calculate using \eqref{eqn:sphereI3} with center $\psi^{-1} (x) = (c_0, c_1, c_2)$ and $r = 1$ that
	\[
		a (g_1^2 + g_2^2) =  2 g_0 + 1.
	\]
However, since $g$ is lightlike, we have
	\[
		a = \frac{2 g_0 + 1}{g_1^2 + g_2^2} = \frac{-g_1^2 - g_2^2}{g_1^2 + g_2^2} = -1,
	\]
so that
	\begin{equation}\label{eqn:normalI3}
		\nu = -g_I = (-g_0 - \tfrac{1}{2}, -g_1, -g_2) = \left(\frac{g_1^2 + g_2^2}{2},  -g_1, -g_2\right).
	\end{equation}
Thus, we can explicitly calculate the unit normal to any spacelike plane given in the coordinate space:
\begin{proposition}\label{prop:norm}
	Let $P$ be a spacelike plane in isotropic $3$-space given via
		\[
			\alpha \mathbf{l} + \beta \mathbf{x} + \gamma \mathbf{y} = q
		\]
	for some $\alpha, \beta, \gamma, q \in \mathbb{R}$.
	Then the unit normal $\nu$ to $P$ is given by
		\[
			\nu = \left(\frac{\beta^2 + \gamma^2}{2\alpha^2}, -\frac{\beta}{\alpha}, -\frac{\gamma}{\alpha}\right).
		\]
\end{proposition}
\begin{proof}
	Since $P$ is a spacelike plane, \eqref{eqn:planeI3} implies $\alpha \neq 0$, allowing us to divide the equation giving the plane by $\alpha$ so that
		\[
			\mathbf{l} + \frac{\beta}{\alpha} \mathbf{x} + \frac{\gamma}{\alpha} \mathbf{y} = \frac{q}{\alpha}.
		\]
	Comparing with \eqref{eqn:planeI3norm}, we see $g_1 = \frac{\beta}{\alpha}$ and $g_2 = \frac{\gamma}{\alpha}$, and substituting these into \eqref{eqn:normalI3} gives the desired conclusion.
\end{proof}

\subsubsection{Surfaces and their contact elements} 
Now let $x : \Sigma \to \mathbb{I}^3$ be a spacelike immersion surface, that is, for every $p \in \Sigma$, the tangent plane $T_{x(p)} x$ is spacelike.
The \emph{Gauss map} of $x$ is then given by taking the unit normal to the tangent plane $T_{x(p)} x$ as in Proposition~\ref{prop:norm}.
(Note that our definition of the Gauss map (almost)\footnote{The only difference arises in the choice of the radius of the unit sphere \cite[Equation~(7.1)]{strubecker_differentialgeometrie_1942}.} coincides with the original definition of \emph{spherical image} in \cite[\S~76]{strubecker_differentialgeometrie_1942} (see also \cite[\S~28]{strubecker_differentialgeometrie_1942-1}).)
In addition, $T_{x(p)} x$ again gives rise to a contact element $L(p) \in \mathbb{Z}$ such that $x(p) = L(p) \cap \mathbb{I}^3$.
Thus, as in the Euclidean case, a surface and its contact element is represented by a map into the set of contact elements, $L: \Sigma \to \mathcal{Z}$, which we refer to as the contact lift of $x$.

Conversely, let $L: \Sigma \to \mathcal{Z}$ be a Legendre immersion into the set of contact elements.
Unlike the Euclidean case, we need to assume some regularity conditions on $L$:
\begin{assumption}
	On every $p \in \Sigma$:
	\begin{enumerate}
		\item $L(p) \cap \mathbb{I}^3 \neq \varnothing$ to ensure that the surface is defined, and
		\item $L(p) \not\subset \mathbb{I}^3$ to ensure that the tangent plane of the surface is spacelike.
	\end{enumerate}
	Note that both of these conditions are trivially satisfied in the Euclidean case.
\end{assumption}

Under the non-degeneracy condition, we can recover the surface $x : \Sigma \to \mathbb{I}^3$ from $L$ by taking the intersection with $\mathbb{I}^3$, that is, $x := L \cap \langle \mathfrak{p} \rangle^\perp$.
We can also recover the Gauss map by considering the unique map into the projective light cone $G: \Sigma \to \mathbb{P}(\mathcal{L})$ so that $G(p) \parallel L(p)$ for every $p \in \Sigma$.
For a hyperplane $\Pi$ parallel to $\mathbb{I}^3$ defined by
	\[
		\Pi := \{ x \in \mathbb{L}^4 : \langle x, \mathfrak{p} \rangle = 1\},
	\]
let $g: \Sigma \to \mathcal{L}$ be defined via $g := G \cap \Pi$.
Using the decomposition $\mathbb{L}^4 =  \mathbb{I}^3 \oplus \langle \tilde{\mathfrak{p}} \rangle$, the condition $\langle g, \mathfrak{p} \rangle = 1$ tells us that
	\[
		g = g_I + \tilde{\mathfrak{p}}.
	\]
The Gauss map $\nu$ is thus given by $\pi_I g = g_I$ while we call $g$ the \emph{lightlike Gauss map} of $x$.

In fact, $L$ being Legendre gives important information about the lightlike Gauss map:
\begin{lemma}\label{lemm:lightlikeGauss}
	Suppose that $x: \Sigma \to \mathbb{I}^3$ is a spacelike immersion, and let $g: \Sigma \to \mathbb{L}$ be its lightlike Gauss map.
	Then we have
		\[
			\langle \dif{x}, g \rangle = 0.
		\]
\end{lemma}
\begin{proof}
	Let $L:\Sigma \to \mathcal{Z}$ be the contact lift of $x$, so that $L$ is a Legendre immersion.
	Take any section $\ell \in \Gamma L$, so that we can write
		\[
			\ell = x + \alpha g
		\]
	for some $\alpha : \Sigma \to \mathbb{R}$.
	Then the contact condition \eqref{eqn:contact} reads
		\[
			0 = \langle \dif{\ell}, g \rangle = \langle \dif{x} + \dif{\alpha} g + \alpha \dif{g}, g \rangle = \langle \dif{x}, g \rangle,
		\]
	giving us the desired conclusion.
\end{proof}

\subsubsection{Isometries in isotropic space -- Hermitian matrix model} 
Finally, the isometries of isotropic $3$-space fixing the origin are given analogously by all $A \in \SOthreeone$ such that $A\mathfrak{p} = \mathfrak{p}$.
Recalling that a spacelike plane $P$ through the origin is defined by $I_{g,0}$ for lightlike $g \not\in \mathbb{I}^3$, it is straightforward to observe that:
\begin{lemma}\label{lemm:spPlanes}
	In isotropic $3$-space, there is only one spacelike plane through the origin up to origin-fixing isometries.
\end{lemma}

On the other hand, such isometries can be described explicitly when using the Hermitian matrix model of isotropic $3$-space, which we explain now.
The set of $2\times2$ Hermitian matrices denoted by $\hermtwo$ is given by
	\[
		\hermtwo := \left\{\begin{pmatrix} x_0 + x_3 & x_1 + i x_2 \\ x_1 - i x_2 & x_0 - x_3 \end{pmatrix} : x_0, x_1, x_2, x_3 \in \mathbb{R} \right\},
	\]
and we identify $\mathbb{L}^4 \cong \hermtwo$ via
	\[
		x = (x_0, x_1, x_2, x_3)^t \sim \begin{pmatrix} x_0 + x_3 & x_1 + i x_2 \\ x_1 - i x_2 & x_0 - x_3 \end{pmatrix} = X.
	\]
Under this identification, we have
	\[
		\langle x, x \rangle = -\det X.
	\]
Thus the action of $F \in \SLtwoC$ on $X \in \hermtwo$ via
	\[
		X \mapsto F X F^*,
	\]
amounts to an origin-fixing isometry of $\mathbb{L}^4$, where $F^*$ denotes the conjugate transpose of $F$.
Since $F$ and $-F$ are the same action, we have that $\SLtwoC$ is a double cover of $\SOthreeone$.

The Hermitian matrix model of isotropic $3$-space is then given by
	\[
		\mathbb{I}^3 := \left\{\begin{pmatrix} 2x_0 & x_1 + i x_2 \\ x_1 - i x_2 & 0 \end{pmatrix} : x_0, x_1, x_2 \in \mathbb{R} \right\}.
	\]
Noting that
	\[
		\mathfrak{p} = (1,0,0,1)^t \sim \begin{pmatrix} 2 & 0 \\ 0 & 0 \end{pmatrix} = \mathfrak{P},
	\]
we have that for $F = \begin{pmatrix} a & b \\ c & d \end{pmatrix} \in \SLtwoC$,
	\[
		F \mathfrak{P} F^*
			=  \begin{pmatrix} 2a\bar{a} & 2a\bar{c} \\ 2\bar{a}c & 2c\bar{c} \end{pmatrix}.
	\]
Thus, if $F$ is an isometry of isotropic $3$-space fixing the origin, that is, $F\mathfrak{P}F^* = \mathfrak{P}$, then we have $c = 0$ and $a = e^{i \theta}$ for some $\theta \in \mathbb{R}$.
Furthermore, since $F \in \SLtwoC$, we must have that
	\[
		1 = ad - bc = e^{i \theta} d
	\]
so that $d = e^{-i \theta}$.
Summarizing:
\begin{lemma}\label{lemm:isometry}
	The action of $F \in \SLtwoC$ on $X \in \hermtwo$ is an isometry of isotropic $3$-space fixing the origin if and only if $F \in \SUoneohone$ where
		\begin{equation}\label{eqn:lieGroup}
			\begin{aligned}
				 \SUoneohone &:= \left\{F \in \SLtwoC : F \begin{psmallmatrix} 1 & 0 \\ 0 & 0 \end{psmallmatrix} F^* = \begin{psmallmatrix} 1 & 0 \\ 0 & 0 \end{psmallmatrix} \right\} \\
					&= \left\{\begin{pmatrix} e^{i \theta} & \alpha + i \beta \\ 0 & e^{-i \theta} \end{pmatrix}: \alpha, \beta, \theta \in \mathbb{R}\right\}.
			\end{aligned}
		\end{equation}
\end{lemma}
With such explicit parametrization of $\SUoneohone$, we can also parametrize the Lie algebra $\suoneohone$:
	\begin{equation}\label{eqn:lieAlgebra}
		\begin{aligned}
			\suoneohone &:= \left\{ A \in \mathrm{GL}(2,\mathbb{C}): \trace{A} = 0, A \begin{psmallmatrix} 1 & 0 \\ 0 & 0 \end{psmallmatrix} + \begin{psmallmatrix} 1 & 0 \\ 0 & 0 \end{psmallmatrix}A^* = 0\right\} \\
				&= \left\{\begin{pmatrix} i a & z \\ 0 & -i a \end{pmatrix} : a \in \mathbb{R}, z \in \mathbb{C} \right\}.
		\end{aligned}
	\end{equation}

We state the following corollary which will be useful later on:
\begin{corollary}\label{cor:conj}
	If $F \in \SUoneohone$, then
		\[
			F^* \begin{pmatrix} 0 & 0 \\ 0 & 1 \end{pmatrix} F =  \begin{pmatrix} 0 & 0 \\ 0 & 1 \end{pmatrix}.
		\]
\end{corollary}
\begin{proof}
	The result follows directly using calculation with explicit parametrization given in \eqref{eqn:lieGroup}.
\end{proof}

\section{Spin transformations and spinor representation}\label{sect:three}
In this section, we consider the spin transformations of conformally immersed spacelike surfaces in isotropic $3$-space, analogously to those of Euclidean $3$-space defined in \cite[Definition~2.1]{kamberov_bonnet_1998}.
After recovering the structure equations using the Minkowski model of $\mathbb{L}^4$, we utilize the Hermitian matrix model to define spin transformations since the group of origin-fixing isometries play an important role in the definition.
Using spin transformations, we recover the spinor representation of conformal immersions in isotropic $3$-space, and obtain the Weierstrass-type representation for minimal surfaces and the Kenmotsu-type representation for (non-zero) constant mean curvature (cmc) surfaces as an application.

\subsection{Basic surface theory of spacelike surfaces}\label{subsect:threeone}
Suppose that $x : \Sigma \to \mathbb{I}^3$ is a spacelike immersion, i.e.\ the induced metric from the ambient space $\mathbb{L}^4$ is Riemannian.
Thus, we may assume that $(u,v) \in \Sigma$ are conformal coordinates, allowing us to introduce a complex structure via $z = u + i v$, and let
	\[
		\partial_z := \tfrac{1}{2}(\partial_u - i \partial_v), \quad \partial_{\bar{z}} := \tfrac{1}{2}(\partial_u + i \partial_v)
	\]
denote the Wirtinger derivatives.
Then the metric of the surface can be written as
	\[
		\dif{s}^2 = e^{2\sigma} (\dif{u}^2 + \dif{v}^2) = e^{2\sigma} \dif{z}\dif{\bar{z}} = 2 \langle x_z, x_{\bar{z}} \rangle \dif{z}\dif{\bar{z}}
	\]
for some function $\sigma: \Sigma \to \mathbb{R}$.

The key difference from the Euclidean case arises in the consideration of the second fundamental form: the Gauss map $\nu$ of the surface in isotropic $3$-space is defined via contact, not metric, and thus, $\langle \dif{x}, \nu \rangle \neq 0$ in general.
To overcome this issue, we take advantage of viewing isotropic $3$-space within Minkowski $4$-space: we treat spacelike surfaces in isotropic $3$-space as spacelike surfaces (of codimension two) in Minkowski $4$-space with a flat normal bundle.
Every fiber of the normal bundle is spanned by a constant vector $\mathfrak{p}$ and the lightlike Gauss map $g : \Sigma \to \mathcal{L}$ of $x$, i.e.\ $\langle g, \mathfrak{p} \rangle = 1$, since Lemma~\ref{lemm:lightlikeGauss} tells us that the fact that $x$ is an immersion implies $\langle \dif{x}, g \rangle = 0$.
In other words, $\mathfrak{p}$ and $g$ constitute a null basis of the fibers of the normal bundle.

Furthermore, since $\mathfrak{p}$ is a constant section of the normal bundle, we only need to consider the lightlike Gauss map $g$ for the second fundamental form; thus, we define the coefficients of the second fundamental form as
	\[
		L := \langle x_{uu}, g \rangle, \quad M := \langle x_{uv}, g \rangle, \quad N := \langle x_{vv}, g \rangle.
	\]

\begin{remark}
	Since $g = \nu + \tilde{\mathfrak{p}}$ for the Gauss map $\nu$, we have
		\[
			\dif{g} = \dif{\nu},
		\]
	so that the coefficients of the second fundamental form can also be calculated from the Gauss map $\nu$ via
	\[
		L = -\langle x_{u}, \nu_u \rangle, \quad M = -\langle x_u, \nu_v \rangle = - \langle x_v, \nu_u \rangle, \quad N = - \langle x_v, \nu_v \rangle.
	\]
\end{remark}

The shape operator $S$ satisfies
	\[
		S = e^{-2\sigma} \begin{pmatrix} L & M \\ M & N \end{pmatrix},
	\]
so that the mean curvature $H$ is
	\begin{equation}\label{eqn:meanC}
		H := \frac{1}{2}\trace S = \frac{1}{2 e^{2\sigma}} (L+N) = 2e^{-2\sigma}\langle x_{z\bar{z}}, g \rangle,
	\end{equation}
while the Hopf differential can be defined by
	\[
		Q \dif{z}^2 := \frac{1}{4}(L - N - 2i M) \dif{z}^2 = \langle x_{zz}, g\rangle \dif{z}^2.
	\]

Now we use $\{x_z, x_{\bar{z}}, \mathfrak{p}, g\}$ as a basis to calculate the Gauss-Weingarten equations:
	\begin{equation}\label{eqn:GW}
		\begin{cases}
			x_{zz} = 2\sigma_z x_z + Q \mathfrak{p} \\
			x_{z\bar{z}} = \frac{1}{2}e^{2\sigma}H \mathfrak{p} \\
			g_z = -H x_z - 2 e^{-2\sigma} Q x_{\bar{z}}.
		\end{cases}
	\end{equation}
Note that the coefficient of $g$ always vanishes; therefore, we can treat the Gauss-Weingarten equations \eqref{eqn:GW} purely within the scope of isotropic $3$-space \cite[\S~72]{strubecker_differentialgeometrie_1942}.
The Gauss equation and the Codazzi equation follow:
	\[
		\sigma_{z\bar{z}} = 0, \quad H_z = 2 e^{-2\sigma} Q_{\bar{z}}.
	\]
	
\begin{remark}
	Noting that the (extrinsic) Gaussian curvature is given by
		\[
			K := \det S = H^2 - 4e^{-4\sigma} Q\bar{Q},
		\]
	we see that the Gauss equation implies that the Gaussian curvature is extrinsic.
	This is a fact noted by Strubecker in \cite[\S~27]{strubecker_differentialgeometrie_1942-1}, where $K$ is called the \emph{relative curvature}.
\end{remark}

\begin{remark}
	The calculation of the structure equations for surfaces in Euclidean space can be carried out similarly: Namely, viewing the surface in Minkowski $4$-space (having codimension two), one can note that the fibers of the normal bundle are spanned by the Euclidean Gauss map $\nu$ and the timelike point sphere complex $\mathfrak{p}$.
	Since $\mathfrak{p}$ is again a constant section of the normal bundle, the second fundamental form is calculated using $\nu$.
	The Gauss-Weingarten equations can also be calculated using the basis $\{x_z, x_{\bar{z}}, \nu, \mathfrak{p}\}$; however, in the Euclidean case, the coefficient of $\mathfrak{p}$ vanishes, resulting in the usual Gauss-Weingarten equations for surfaces in Euclidean space expressed in terms of $\{x_z, x_{\bar{z}}, \nu\}$.
\end{remark}

For spacelike surfaces in isotropic $3$-space, the coordinate function for the vertical direction is directly related to the mean curvature:
\begin{lemma}
	Suppose that a spacelike surface $x: \Sigma \to \mathbb{I}^3$ is parametrized via
		\[
			x(u,v) = (\mathbf{l}(u,v), \mathbf{x}(u,v), \mathbf{y}(u,v), \mathbf{l}(u,v))^t.
		\]
	Then we have
		\[
			\mathbf{l}_{z \bar{z}} = \frac{1}{2}e^{2\sigma}H.
		\]
\end{lemma}
\begin{proof}
	Note that since $\langle x, \tilde{\mathfrak{p}} \rangle = \mathbf{l}$, we have that $\mathbf{l}_{z \bar{z}} = \langle x_{z\bar{z}} , \tilde{\mathfrak{p}} \rangle$.
	Thus,
		\[
			\mathbf{l}_{z \bar{z}}= \frac{1}{4}\langle x_{uu} + x_{vv}, \tilde{\mathfrak{p}} \rangle = \frac{1}{4}(L+N),
		\]
	giving us the desired conclusion.
\end{proof}

Isometries fixing the origin are central to the definition of spin transformations \cite[Definition~2.1]{kamberov_bonnet_1998}; thus, we will use the Hermitian matrix model of $\mathbb{I}^3$ as the switch allows us to write the isometries explicitly as in Lemma~\ref{lemm:isometry}.
For this, we first rewrite the Gauss-Weingarten equations \eqref{eqn:GW} within the context of Hermitian matrix model.

Choosing the basis of $\mathbb{L}^4$ as $\{ e_1, e_2, \mathfrak{p}, \tilde{\mathfrak{p}}\}$ for
	\[
		e_1 := (0, 1 ,0, 0)^t, \quad e_2 := (0, 0, 1, 0)^t,
	\]
we see that they correspond to
	\[
		E_1 = \begin{pmatrix} 0 & 1 \\ 1 & 0 \end{pmatrix},
			\quad E_2 = \begin{pmatrix} 0 & i \\ -i & 0 \end{pmatrix},
			\quad \mathfrak{P} = \begin{pmatrix} 2 & 0 \\ 0 & 0 \end{pmatrix},
			\quad \tilde{\mathfrak{P}} = \begin{pmatrix} 0 & 0 \\ 0 & -1 \end{pmatrix},
	\]
in the Hermitian matrix model, respectively.

Now, let $F : \Sigma \to \SUoneohone$ such that
	\begin{equation}\label{eqn:frame}
		X_u = e^\sigma F E_1 F^* \quad\text{and}\quad X_v = e^\sigma F E_2 F^*.
	\end{equation}
Since $F \subset \SUoneohone$ implies that
	\[
		\mathfrak{P} = F \mathfrak{P} F^*,
	\]
we then have
	\[
		G = F \tilde{\mathfrak{P}} F^*,
	\]
where $G$ is the lightlike Gauss map in the Hermitian matrix model.

Thus, using $X_z = \frac{1}{2}(X_u - i X_v) = e^\sigma F \frac{E_1 - i E_2}{2} F^*$, we may rewrite the structure equations \eqref{eqn:GW} as
	\begin{equation}\label{eqn:GW2}
		\begin{cases}
			X_{zz} = F \left(2 \sigma_z e^\sigma \frac{E_1 - i E_2}{2} + Q \mathfrak{P}\right) F^* = F \begin{psmallmatrix} 2Q & 2\sigma_z e^\sigma \\ 0 & 0 \end{psmallmatrix} F^* \\
			X_{z\bar{z}} = F\left(\frac{1}{2}e^{2\sigma} H \mathfrak{P}\right) F^* = F \begin{psmallmatrix} e^{2\sigma} H & 0 \\ 0 & 0 \end{psmallmatrix} F^* \\
			G_z = F \left(- e^\sigma H \frac{E_1 - i E_2}{2} - 2 e^{-\sigma} Q \frac{E_1 + i E_2}{2}\right) F^* = F \begin{psmallmatrix} 0 & -e^\sigma H \\ -2 e^{-\sigma} Q & 0 \end{psmallmatrix} F^*.
		\end{cases}
	\end{equation}
On the other hand, differentiating $X_z$ and $G$, we also obtain that
	\begin{equation}\label{eqn:GW3}
		\begin{cases}
			X_{zz} = e^\sigma F \left(\sigma_z  \frac{E_1 - i E_2}{2} + F^{-1}F_z  \frac{E_1 - i E_2}{2} +  \frac{E_1 - i E_2}{2} (F^{-1}F_{\bar{z}})^* \right) F^* \\
			X_{z\bar{z}} = e^\sigma F \left(\sigma_{\bar{z}}  \frac{E_1 - i E_2}{2} + F^{-1}F_{\bar{z}}  \frac{E_1 - i E_2}{2} +  \frac{E_1 - i E_2}{2} (F^{-1}F_z)^* \right) F^* \\
			G_z = F(F^{-1} F_z \tilde{\mathfrak{P}} + \tilde{\mathfrak{P}} (F^{-1} F_{\bar{z}})^*)F^*.
		\end{cases}
	\end{equation}
Since $F^{-1}\dif{F}$ is an $\suoneohone$-valued $1$--form, we use \eqref{eqn:lieAlgebra} to write
	\[
		F^{-1} F_z =: \begin{pmatrix} A & B \\ 0 & -A \end{pmatrix}, \quad F^{-1} F_{\bar{z}} =: \begin{pmatrix} a & b \\ 0 & -a \end{pmatrix}
	\]
for some $A, B, a, b : \Sigma \to \mathbb{C}$.
Now comparing the different expressions for $X_{zz}, X_{z\bar{z}}, G_z$ in \eqref{eqn:GW2} and \eqref{eqn:GW3}, we obtain the relations
	\begin{gather*}
		\begin{pmatrix} 2e^{-\sigma}Q & 2\sigma_z \\ 0 & 0 \end{pmatrix} = \begin{pmatrix} \bar{b} & \sigma_z + A - \bar{a} \\ 0 & 0 \end{pmatrix},
			\quad \begin{pmatrix} e^{\sigma} H & 0 \\ 0 & 0 \end{pmatrix} = \begin{pmatrix} \bar{B} & \sigma_{\bar{z}} + a - \bar{A} \\ 0 & 0 \end{pmatrix} \\
			\begin{pmatrix} 0 & - e^\sigma H \\ -2 e^{-\sigma} Q & 0 \end{pmatrix} = \begin{pmatrix} 0 & -B \\ -\bar{b} & A + \bar{a} \end{pmatrix},
	\end{gather*}
allowing us to solve for $A, B, a, b$, and conclude that
	\begin{equation}\label{eqn:GWHerm}
		F^{-1} F_z =: \begin{pmatrix} \frac{1}{2}\sigma_z & e^\sigma H \\ 0 & -\frac{1}{2}\sigma_z \end{pmatrix}, \quad F^{-1} F_{\bar{z}} =: \begin{pmatrix} -\frac{1}{2}\sigma_{\bar{z}} & 2 e^{-\sigma} \bar{Q} \\ 0 & \frac{1}{2}\sigma_{\bar{z}} \end{pmatrix}.
	\end{equation}

\subsection{Spin transformation and Dirac--type operator}\label{subsect:threetwo}
Spin transformations relate two conformally equivalent surfaces via homotheties and rotations of the corresponding tangent planes.
Since rotations fixing the origin are given by $\SUoneohone$, we let
	\[
		\mathcal{G} := \mathbb{R}^+ \otimes \SUoneohone
	\]
where $\mathbb{R}^+$ denotes the set of positive real numbers.
We define spin transformations of a conformal immersion as follows:
\begin{definition}[{cf.\ \cite[Definition~2.1]{kamberov_bonnet_1998}}]\label{def:spin}
	Let $X, \tilde{X}: \Sigma \to \mathbb{I}^3$ be conformal immersions.
	Then $\tilde{X}$ is called a \emph{spin transformation} of $X$ if there exists some $B: \Sigma \to \mathcal{G}$ such that
		\begin{equation}\label{eqn:spin}
			\dif{\tilde{X}} = B \dif{X} B^*.
		\end{equation}
\end{definition}

Suppose that we wish to find spin transformations $\tilde{X}$ of $X$ using \eqref{eqn:spin}.
For such $\tilde{X}$ to exist, we need $\dif{(\dif{\tilde{X}})} = 0$, or equivalently using coordinates,
	\[
		\tilde{X}_{z\bar{z}} = \tilde{X}_{\bar{z}z}.
	\]
To see when $\dif{\tilde{X}}$ is closed, note that for $\hat{F} := B F$, \eqref{eqn:spin} implies
	\[
		\tilde{X}_z = B X_z B^* = e^{\sigma} \hat{F} \tfrac{E_1 - i E_2}{2} \hat{F}^*, \quad \tilde{X}_{\bar{z}} = B X_{\bar{z}} B^* = e^{\sigma} \hat{F} \tfrac{E_1 + i E_2}{2} \hat{F}^*,
	\]
allowing us to calculate
	\begin{equation}\label{eqn:newComp}
		\begin{aligned}
			\tilde{X}_{z\bar{z}} &= e^\sigma \hat{F} \left(\sigma_{\bar{z}}  \tfrac{E_1 - i E_2}{2} + \hat{F}^{-1}\hat{F}_{\bar{z}}  \tfrac{E_1 - i E_2}{2} +  \tfrac{E_1 - i E_2}{2} (\hat{F}^{-1}\hat{F}_z)^* \right) \hat{F}^* \\
			\tilde{X}_{\bar{z}z} &= e^\sigma \hat{F} \left(\sigma_z  \tfrac{E_1 + i E_2}{2} + \hat{F}^{-1}\hat{F}_z  \tfrac{E_1 + i E_2}{2} +  \tfrac{E_1 + i E_2}{2} (\hat{F}^{-1}\hat{F}_{\bar{z}})^* \right) \hat{F}^*.
		\end{aligned}
	\end{equation}
Defining
	\[
		\Omega := \begin{pmatrix} \Omega_{11} & \Omega_{12} \\ \Omega_{21} & \Omega_{22} \end{pmatrix} := F^{-1} B^{-1} B_z F,
			\quad \Lambda := \begin{pmatrix} \Lambda_{11} & \Lambda_{12} \\ \Lambda_{21} & \Lambda_{22} \end{pmatrix} := F^{-1} B^{-1} B_{\bar{z}} F,
	\]
we can verify that
	\[
		\hat{F}^{-1}\hat{F}_z = \Omega + F^{-1}F_z, \quad \hat{F}^{-1}\hat{F}_{\bar{z}} = \Lambda + F^{-1}F_{\bar{z}},
	\]
so that \eqref{eqn:newComp} implies that $\dif{\tilde{X}}$ is closed if and only if
	\[
		\begin{pmatrix} \overline{\Omega_{12}} + e^{\sigma} H & \Lambda_{11} + \overline{\Omega_{22}} \\ 0 & \Lambda_{21} \end{pmatrix}
			= \begin{pmatrix} \Omega_{12} + e^{\sigma} H & 0 \\ \overline{\Lambda_{11}} + \Omega_{22} & \overline{\Lambda_{21}} \end{pmatrix},
	\]
or equivalently,
	\begin{equation}\label{eqn:compatibility}
		\Omega_{12} \in \mathbb{R}, \quad \Lambda_{21} \in \mathbb{R}, \quad \overline{\Lambda_{11}} + \Omega_{22} = 0.
	\end{equation}

Now suppose that $B: \Sigma \to \mathcal{G}$ such that the compatibility condition \eqref{eqn:compatibility} is satisfied.
To see how the metric changes, we note that
	\begin{equation}\label{eqn:metric}
		e^{2\tilde{\sigma}}\dif{z}\dif{\bar{z}} =: \langle \dif{\tilde{x}}, \dif{\tilde{x}} \rangle = -\det{\dif{\tilde{X}}} = -\det{(B \dif{X} B^*)} = (\det{B})^2 e^{2\sigma}\dif{z}\dif{\bar{z}}.
	\end{equation}
	
Now, for some $\tilde{F}: \Sigma \to \SUoneohone$, we can write 
	\[
		\tilde{X}_z = e^{\tilde{\sigma}} \tilde{F} \tfrac{E_1 - i E_2}{2} \tilde{F}^*,
	\]
so that $\tilde{X}$ satisfies the analogous Gauss-Weingarten equations \eqref{eqn:GWHerm}:
	\begin{equation}\label{eqn:newGW}
		\tilde{F}^{-1} \tilde{F}_z =: \begin{pmatrix} \frac{1}{2}\tilde{\sigma}_z & e^{\tilde{\sigma}} \tilde{H} \\ 0 & -\frac{1}{2}\tilde{\sigma}_z \end{pmatrix}, \quad \tilde{F}^{-1} \tilde{F}_{\bar{z}} =: \begin{pmatrix} -\frac{1}{2}\tilde{\sigma}_{\bar{z}} & 2 e^{-\tilde{\sigma}} \bar{\tilde{Q}} \\ 0 & \frac{1}{2}\tilde{\sigma}_{\bar{z}} \end{pmatrix}
	\end{equation}
where $\tilde{H}$ and $\tilde{Q}$ are the mean curvature and Hopf differential factor of $\tilde{X}$, respectively.
On the other hand, if we let $A := fB : \Sigma \to \SUoneohone$ for some $f: \Sigma \to \mathbb{R}^+$, then we have
	\[
		\tilde{F} = A F = fBF.
	\]
Thus we also have
	\begin{equation}\label{eqn:newGW2}
		\begin{aligned}
			\tilde{F}^{-1} \tilde{F}_z &= F^{-1} B^{-1} B_z F + f^{-1}f_z I + F^{-1}F_z, \\
			\tilde{F}^{-1} \tilde{F}_{\bar{z}} &= F^{-1} B^{-1} B_{\bar{z}} F + f^{-1}f_{\bar{z}} I + F^{-1}F_{\bar{z}}.
		\end{aligned}
	\end{equation}
Noting that since $\det{A} = 1$, the change in metric \eqref{eqn:metric} implies
	\[
		f^2 = \frac{1}{\det{B}} = e^{\sigma - \tilde{\sigma}},
	\]
so that
	\[
		f^{-1} f_z = \frac{(f^2)_z}{2f^2} = \tfrac{1}{2}(\sigma_z - \tilde{\sigma}_z).
	\]
Therefore, comparing the two expressions for the logarithmic derivative of $\tilde{F}$ \eqref{eqn:newGW} and \eqref{eqn:newGW2}, we obtain
	\begin{align*}
		\begin{pmatrix} \frac{1}{2}\tilde{\sigma}_z & e^{\tilde{\sigma}} \tilde{H} \\ 0 & -\frac{1}{2}\tilde{\sigma}_z \end{pmatrix}
			&= \begin{pmatrix} \Omega_{11} + \sigma_z - \frac{1}{2}\tilde{\sigma}_z & \Omega_{12} + e^\sigma H \\ \Omega_{21} & \Omega_{22}- \frac{1}{2} \tilde{\sigma}_z \end{pmatrix} \\
		\begin{pmatrix} -\frac{1}{2}\tilde{\sigma}_{\bar{z}} & 2 e^{-\tilde{\sigma}} \bar{\tilde{Q}} \\ 0 & \frac{1}{2}\tilde{\sigma}_{\bar{z}} \end{pmatrix}
			&=  \begin{pmatrix} \Lambda_{11} - \frac{1}{2}\tilde{\sigma}_{\bar{z}} & \Lambda_{12} + 2 e^{-\sigma} \bar{Q} \\ \Lambda_{21}  & \Lambda_{22} + \sigma_{\bar{z}} - \frac{1}{2} \tilde{\sigma}_{\bar{z}}\end{pmatrix},
	\end{align*}
so that
	\begin{equation}\label{eqn:compatibility2}
		\Omega_{21} = \Omega_{22} = \Lambda_{11} = \Lambda_{21} = 0, \quad \Omega_{11} = \overline{\Lambda_{22}},
	\end{equation}
while
	\[
		\tilde{H} = \frac{\Omega_{12} + e^{\sigma} H}{ e^\sigma \det{B}} = \frac{H + \rho}{\det{B}}
	\]
for
	\begin{equation}\label{eqn:compatibility3}
		\rho = e^{-\sigma} \Omega_{12}.
	\end{equation}
Finally, we also note the change in the Hopf differential:
	\[
		\tilde{Q} = \frac{1}{2}\det B (e^{\sigma}\overline{\Lambda_{12}} + 2 Q).
	\]

The compatibility condition on $\Omega$ and $\Lambda$  \eqref{eqn:compatibility}, and the subsequent equations \eqref{eqn:compatibility2}, \eqref{eqn:compatibility3} can be expressed in terms of a single Dirac--type equation for isotropic $3$--space:
\begin{theorem}[{cf.\ \cite[Lemmata~2.3, 2.4]{kamberov_bonnet_1998}}]\label{thm:Dirac}
	Let $X: \Sigma \to \mathbb{I}^3$ be a conformal immersion with mean curvature $H$.
	A spin transformation $\tilde{X}: \Sigma \to \mathbb{I}^3$ of $X$ is well-defined if and only if $B : \Sigma \to \mathcal{G} = \mathbb{R}^+ \otimes \SUoneohone$ satisfies the Dirac--type equation for isotropic $3$--space, that is, for some $\rho : \Sigma \to \mathbb{R}$,
		\begin{equation}\label{eqn:Dirac}
			B^{-1} \dif{B} \wedge \dif{X} = -\rho \dif{X} \wedge \tilde{\mathfrak{P}} \dif{X},
		\end{equation}
	where $\tilde{\mathfrak{P}} = \begin{psmallmatrix} 0 & 0 \\ 0 & -1 \end{psmallmatrix}$.
	Furthermore, the mean curvature $\tilde{H}$ of $\tilde{X}$ can be obtained via
		\begin{equation}\label{eqn:newMean}
			\tilde{H} = \frac{H + \rho}{\det{B}}.
		\end{equation}
\end{theorem}
\begin{proof}
	We only need to see that the Dirac--type equation \eqref{eqn:Dirac} is equivalent to the compatibility condition for $\tilde{X}$ \eqref{eqn:compatibility}.
	To do this, first we calculate using coordinates that
		\begin{align*}
			B^{-1} \dif{B} \wedge \dif{X}
				&= (B^{-1} B_z X_{\bar{z}} - B^{-1} B_{\bar{z}} X_z)(\dif{z} \wedge \dif{\bar{z}}) \\
				&= e^\sigma F (\Omega \tfrac{E_1 + i E_2}{2} - \Lambda  \tfrac{E_1 - i E_2}{2}) F^*(\dif{z} \wedge \dif{\bar{z}}) \\
				&= e^\sigma F \begin{pmatrix} \Omega_{12} & -\Lambda_{11} \\ \Omega_{22} & -\Lambda_{21} \end{pmatrix} F^*(\dif{z} \wedge \dif{\bar{z}}).
		\end{align*}
	On the other hand, we also have using Corollary~\ref{cor:conj},
		\begin{align*}
			 -\rho \dif{X} \wedge \tilde{\mathfrak{P}} \dif{X}
				&= -\rho (X_z \dif{z} + X_{\bar{z}} \dif{\bar{z}}) \wedge (\tilde{\mathfrak{P}}X_z \dif{z} + \tilde{\mathfrak{P}}X_{\bar{z}} \dif{\bar{z}}) \\
				&=  -\rho e^{2\sigma} F  ( \tfrac{E_1 - i E_2}{2} \tilde{\mathfrak{P}}  \tfrac{E_1 + i E_2}{2} -  \tfrac{E_1 + i E_2}{2} \tilde{\mathfrak{P}} \tfrac{E_1 - i E_2}{2}) F^*(\dif{z} \wedge \dif{\bar{z}}) \\
				&= -\rho e^{2\sigma} F \begin{pmatrix} -1 & 0 \\ 0 & 0 \end{pmatrix} F^*(\dif{z} \wedge \dif{\bar{z}}).
		\end{align*}
	Thus the Dirac--type equation is satisfied if and only if
		\begin{equation}\label{eqn:DiracDirty}
			\begin{pmatrix} \Omega_{12} & -\Lambda_{11} \\ \Omega_{22} & -\Lambda_{21} \end{pmatrix} =  \begin{pmatrix} \rho e^\sigma & 0 \\ 0 & 0 \end{pmatrix}.
		\end{equation}
	
	Now, for one direction of the proof, suppose the compatibility condition \eqref{eqn:compatibility} is satisfied so that $\tilde{X}$ is well-defined.
	Thus, we have the relations \eqref{eqn:compatibility2} and \eqref{eqn:compatibility3}, which imply that \eqref{eqn:DiracDirty} holds, that is, $B$ satisfies the Dirac--type equation \eqref{eqn:Dirac}.
	
	For the other direction, if $B$ satisfies the Dirac--type equation \eqref{eqn:Dirac}, then \eqref{eqn:DiracDirty} holds, which implies that the compatibility condition \eqref{eqn:compatibility} is satisfied.
\end{proof}

\subsection{Spinor representation}
We have now gathered the necessary ingredients to obtain the spinor representation of conformal spacelike surfaces in isotropic $3$-space.
By Lemma~\ref{lemm:spPlanes}, we can map any spacelike plane through the origin to another one via isometries fixing the origin.
Thus, any conformal spacelike immersion can be obtained as a spin transformation of the map $X: \Sigma \to \mathbb{I}^3$ where
	\[
		X(u,v) := \begin{pmatrix} 0 & u + i v \\ u - i v & 0 \end{pmatrix} = \begin{pmatrix} 0 & z \\ \bar{z} & 0 \end{pmatrix}.
	\]
Then we have that $F: \Sigma \to \SUoneohone$ as in \eqref{eqn:frame} is the identity map.

Now to consider spin transfomations of $X$, write $B: \Sigma \to \mathcal{G}$ explicitly as
	\[
		B = \begin{pmatrix} \alpha & \bar{\beta} \\ 0 & \bar{\alpha} \end{pmatrix},
	\]
for some complex functions $\alpha, \beta$ defined on $\Sigma$, and note that $\alpha$ never vanishes since $\det B \neq 0$ everywhere.
Then we calculate
	\begin{align*}
		\Omega &= F^{-1} B^{-1} B_z F = B^{-1} B_z = \frac{1}{|\alpha|^2}
			\begin{pmatrix}
				\bar{\alpha} \alpha_z & \bar{\alpha} \bar{\beta}_z - \bar{\beta}\bar{\alpha}_z \\
				0 & \alpha \bar{\alpha}_z
			\end{pmatrix} \\
		\Lambda &= F^{-1} B^{-1} B_{\bar{z}} F = B^{-1} B_{\bar{z}} = \frac{1}{|\alpha|^2}
			\begin{pmatrix}
				\bar{\alpha} \alpha_{\bar{z}} & \bar{\alpha} \bar{\beta}_{\bar{z}} - \bar{\beta}\bar{\alpha}_{\bar{z}} \\
				0 & \alpha \bar{\alpha}_{\bar{z}}
			\end{pmatrix}.
	\end{align*}
	
The map $B$ must satisfy the Dirac--type equation; using equivalent conditions in \eqref{eqn:DiracDirty}, we find that since $\alpha$ never vanishes
	\[
		\alpha_{\bar{z}} = 0, \quad \alpha \beta_{\bar{z}} \in \mathbb{R}.
	\]
For such $B$, we can then calculate
	\[
		\tilde{X}_z = B X_z B^*
			= \begin{pmatrix} \alpha & \bar{\beta} \\ 0 & \bar{\alpha} \end{pmatrix}
				\begin{pmatrix} 0 & 1 \\ 0 & 0 \end{pmatrix}
				\begin{pmatrix} \bar{\alpha} & 0 \\ \beta & \alpha \end{pmatrix}
			= \begin{pmatrix} \alpha \beta & \alpha^2 \\ 0 & 0 \end{pmatrix},
	\]
or in terms of coordinates,
	\begin{equation}\label{eqn:confDeri}
		\tilde{x}_z = \frac{1}{2}\left(\frac{\beta}{\alpha}, 1, -i \right) \alpha^2.
	\end{equation}
Summarizing:
\begin{theorem}[Spinor representation of conformal immersions]\label{thm:spinor}
	Any conformal immersion $\tilde{x}: \Sigma \to \mathbb{I}^3$ can locally be represented as
		\begin{equation}\label{eqn:representation}
			\tilde{x} =  \frac{1}{2} \Re \int \left(\frac{\beta}{\alpha}, 1, -i \right) \alpha^2 \dif{z},
		\end{equation}
	for some non-vanishing holomorphic function $\alpha: \Sigma \to \mathbb{C}$, and a complex valued function $\beta: \Sigma \to \mathbb{C}$ satisfying
		\begin{equation}\label{eqn:compatRep}
			\alpha \beta_{\bar{z}} - \bar{\alpha} \bar{\beta}_z = 0.
		\end{equation}
	The metric of $\tilde{x}$ then is
		\[
			\dif{s}^2 = |\alpha|^4 \dif{z}\dif{\bar{z}},
		\]
	while the mean curvature $\tilde{H}$ of $\tilde{x}$ is given by
		\begin{equation}\label{eqn:newMean2}
			\tilde{H} = \frac{\beta_{\bar{z}}}{\bar{\alpha}|\alpha|^2}.
		\end{equation}
	Moreover, the lightlike Gauss map $\tilde{g}: \Sigma \to \mathcal{L}$ is given by
		\begin{equation}\label{eqn:newGauss2}
			\tilde{g} = -\tfrac{1}{2}\left(1 + |\tfrac{\beta}{\alpha}|^2, 2 \Re \tfrac{\beta}{\alpha}, -2 \Im \tfrac{\beta}{\alpha}, -1 + |\tfrac{\beta}{\alpha}|^2\right)^t.
		\end{equation}
\end{theorem}

\begin{proof}
	Since the representation by the integral formula \eqref{eqn:representation} is already proven, we proceed to prove the rest of the claims.
	For the metric, we use the relation on metric \eqref{eqn:metric} to see that
		\[
			\dif{\tilde{s}}^2 = \det B^2 \dif{s}^2 = |\alpha|^4 \dif{z}\dif{\bar{z}},
		\]
	while for the mean curvature $\tilde{H}$, we use \eqref{eqn:newMean} with the value of $\rho$ given by \eqref{eqn:compatibility3}:
		\[
			\tilde{H} = \frac{\rho}{\det B} = \frac{\beta_{\bar{z}} }{\bar{\alpha}|\alpha|^2}.
		\]
	For the lightlike Gauss map, first note that if $A = f B : \Sigma \to \SUoneohone$ for some $f : \Sigma \to \mathbb{R}^+$, then 
		\[
			A = \tfrac{1}{|\alpha|} B.
		\]
	Thus the new lightlike Gauss map is given by
		\[
			\tilde{G} = A \tilde{\mathfrak{P}} A^*
				= \frac{1}{|\alpha|^2} \begin{pmatrix} \alpha & \bar{\beta} \\ 0 & \bar{\alpha} \end{pmatrix}
					\begin{pmatrix} 0 & 0 \\ 0 & -1 \end{pmatrix}
					\begin{pmatrix} \bar{\alpha} & 0 \\ \beta & \alpha \end{pmatrix}
				= - \frac{1}{|\alpha|^2}\begin{pmatrix} |\beta|^2 & \alpha \bar{\beta} \\ \bar{\alpha}\beta & |\alpha|^2 \end{pmatrix}.
		\]
	Since $\tilde{G} \in \hermtwo$, we can let $\tilde{g}$ be the corresponding vector in $\mathbb{L}^4$ to obtain the desired conclusion.
\end{proof}

\begin{remark}\label{rem:kenmotsu}
	We remark here that one can also prescribe the mean curvature function $H : \Sigma \to \mathbb{R}$ first, and find holomorphic function $\alpha$ and a complex function satisfying \eqref{eqn:newMean2}.
	Such $\alpha$ and $\beta$ will satisfy the compatibility condition \eqref{eqn:compatRep}, thus giving us the \emph{Kenmotsu representation} of conformal surfaces with prescribed mean curvature \cite[Theorem~2]{kenmotsu_weierstrass_1979}.
\end{remark}

	To check the robustness of the representation, we can verify all the claims of Theorem~\ref{thm:spinor} directly as follows.
	Let $x: \Sigma \to \mathbb{I}^3$ be defined via \eqref{eqn:representation}.
	To calculate the metric, first note that with $x_z$ as in \eqref{eqn:confDeri}, we have
		\[
			\langle x_z, x_z \rangle
				= \tfrac{\alpha^4}{4}  \Big\langle \left(\tfrac{\beta}{\alpha}, 1, -i, \tfrac{\beta}{\alpha} \right)^t, \left(\tfrac{\beta}{\alpha}, 1, -i, \tfrac{\beta}{\alpha} \right)^t \Big\rangle
				= 0
		\]
	while using $x_{\bar{z}} = \frac{1}{2}\left(\frac{\bar{\beta}}{\bar{\alpha}}, 1, i \right) \bar{\alpha}^2$, we have
		\[
			2 \langle x_z, x_{\bar{z}} \rangle
				= \tfrac{|\alpha|^4}{2} \Big\langle \left(\tfrac{\beta}{\alpha}, 1, -i, \tfrac{\beta}{\alpha} \right), \left(\tfrac{\bar{\beta}}{\bar{\alpha}}, 1, i, \tfrac{\bar{\beta}}{\bar{\alpha}} \right) \Big\rangle
				= |\alpha|^4.
		\]
Thus the metric of the given surface $x$ is
	\[
		\dif{s}^2 = 2 \langle x_z, x_{\bar{z}} \rangle  \dif{z}\dif{\bar{z}} = |\alpha|^4 \dif{z}\dif{\bar{z}}.
	\]

On the other hand, with $g$ given as in \eqref{eqn:newGauss2}, we can check that
	\[
		\langle g, g \rangle = 0, \quad \langle g, \mathfrak{p} \rangle = 1, \quad \langle x_z, g \rangle = 0,
	\]
so that $g$ is the lightlike Gauss map of $x$.
	
Finally, since $\alpha$ is holomorphic, we have
	\[
		x_{z\bar{z}}
			= \tfrac{1}{2}(\alpha \beta, \alpha^2, -i \alpha^2, \alpha \beta)^t_{\bar{z}}
			= \tfrac{1}{2}(\alpha \beta_{\bar{z}}, 0, 0, \alpha \beta_{\bar{z}})^t,
	\]
and using the formula for the mean curvature \eqref{eqn:meanC}, we calculate
	\begin{align*}
		H
			&= \tfrac{2}{|\alpha|^4} \langle x_{z\bar{z}}, g \rangle \\
			&= -\tfrac{1}{2|\alpha|^4} \langle \left(\alpha \beta_{\bar{z}}, 0, 0, \alpha \beta_{\bar{z}}\right)^t, \left(1 + |\tfrac{\beta}{\alpha}|^2, 2 \Re \tfrac{\beta}{\alpha}, -2 \Im \tfrac{\beta}{\alpha}, -1 + |\tfrac{\beta}{\alpha}|^2\right)^t \rangle \\
			&= \frac{\beta_{\bar{z}}}{|\alpha|^2\bar{\alpha}}.
	\end{align*}

\subsection{Weierstrass-type representation and Kenmotsu-type representation for cmc surfaces}
Now suppose that the conformal immersion is minimal, i.e.\ $H = 0$, so that \eqref{eqn:newMean2} implies that $\beta_{\bar{z}} = 0$.
Since $\alpha$ and $\beta$ are now both holomorphic, define a meromorphic function $h: \Sigma \to \mathbb{C}$ and a holomorphic $1$--form $\omega$ via
	\[
		h := \frac{\beta}{\alpha}, \quad \omega := \alpha^2 \dif{z}.
	\]
Then the Weierstrass-type representation for minimal surfaces in isotropic $3$-space \cite[Equation~8.31]{strubecker_differentialgeometrie_1942} is obtained as an application of the spinor representation:
\begin{theorem}[Weierstrass-type representation of minimal immersions]\label{thm:Weierstrass}
	Any minimal immersion $x: \Sigma \to \mathbb{I}^3$ can locally be represented as
		\[
			x =  \frac{1}{2} \Re \int \left(h, 1, -i \right) \omega,
		\]
	for some meromorphic $h: \Sigma \to \mathbb{C}$ and holomorphic $1$--form $\omega$ such that $h^2 \omega$ is holomorphic.
	
	The metric of $x$ is then given by
		\[
			\dif{s}^2 = |\omega|^2,
		\]
	the Hopf differential
		\[
			Q = \frac{1}{2}\omega \dif{h},
		\]
	and the lightlike Gauss map $g: \Sigma \to \mathcal{L}$
		\[
			g = -\tfrac{1}{2}\left(1 + |h|^2, 2 \Re h, -2 \Im h, -1 + |h|^2\right).
		\]
\end{theorem}

On the other hand, suppose that the conformal immersion has non-zero cmc $H \neq 0$.
Then by \eqref{eqn:newMean2} we have
	\[
		\beta_{\bar{z}} = H \bar{\alpha}^2 \alpha
	\]
for some holomorphic $\alpha: \Sigma \to \mathbb{C}$.
Thus,
	\[
		\beta = H \alpha \int \bar{\alpha}^2 \dif{\bar{z}} + \gamma
	\]
for some holomorphic $\gamma: \Sigma \to \mathbb{C}$.
Therefore, if we define $h_1, h_2: \Sigma \to \mathbb{C}$ via
	\[
		h_1 := H \int \alpha^2 \dif{z}, \quad h_2 := \frac{\gamma}{\alpha},
	\]
then we have that $h_1$ is holomorphic, while $h_2$ is meromorphic, satisfying
	\[
		\frac{\beta}{\alpha} = \bar{h}_1 + h_2.
	\]
Thus the Kenmotsu-type representation for non-zero constant mean curvature $H$ surfaces can also be obtained as an application of the spinor representation:
\begin{theorem}[Kenmotsu-type representation of cmc immersions]\label{thm:Kenmotsu}
	Any non-zero cmc $H$ immersion $x: \Sigma \to \mathbb{I}^3$ can locally be represented as
		\[
			x =  \frac{1}{2} \Re \int \left(\bar{h}_1 + h_2, 1, -i \right) \omega,
		\]
	for some holomorphic $h_1: \Sigma \to \mathbb{C}$, meromorphic $h_2: \Sigma \to \mathbb{C}$ and holomorphic $1$--form $\omega$ such that $h_2^2 \omega$ is holomorphic while
		\begin{equation}\label{eqn:h1cond}
			\dif{h}_1 = H\omega.
		\end{equation}
	
	The metric of $x$ is then given by
		\[
			\dif{s}^2 = |\omega|^2,
		\]
	the Hopf differential
		\[
			Q = \frac{1}{2}\omega  \dif{h_2} ,
		\]
	and the lightlike Gauss map $g: \Sigma \to \mathcal{L}$
		\[
			g = -\tfrac{1}{2}\left(1 + |h|^2, 2 \Re h, -2 \Im h, -1 + |h|^2\right)
		\] 
	for $h := \bar{h}_1 + h_2$.
	We call $(h_2, \omega)$ the Kenmotsu data.
\end{theorem}

\begin{remark}
	We note here that the Weierstrass-type representation for minimal surfaces in Theorem~\ref{thm:Weierstrass} is a special case of the Kenmotsu-type representation for cmc surfaces in Theorem~\ref{thm:Kenmotsu}, as \eqref{eqn:h1cond} says that if $H = 0$, then $h_1 \equiv \text{constant}$, and $h:= \bar{h}_1 + h_2$ is meromorphic with holomorphic $h^2 \omega$.
\end{remark}

\subsection{Examples of cmc surfaces}\label{subsect:examples}
In this section, we compute some cmc surfaces explicitly using the Kenmotsu-type representation in Theorem~\ref{thm:Kenmotsu}.

\begin{description}[style=unboxed,leftmargin=0cm]
	\item[Spheres]
		For some constant $H$, let $\omega = e^z \dif{z}$ and $h_2 = 0$.
	Then
		\[
			h_1 = H \int e^z \dif{z} = H e^z.
		\]
	Then by the Kenmotsu-type representation, we have
		\[
			x_z = \frac{1}{2}(H e^{2u}, e^z, -i e^z), \quad x_{\bar{z}} = \frac{1}{2} (H e^{2 u},  e^{\bar {z}}, i e^{\bar {z}}),
		\]
	so that
		\begin{align*}
			 x_u &= x_z + x_{\bar{z}} = \frac{1}{2}(2 H e^{2u}, e^z + e^{\bar {z}}, -i (e^z - e^{\bar {z}})) = (H e^{2u}, e^u \cos v, e^u \sin v),  \\
			 x_v &= i(x_z - x_{\bar{z}}) = \frac{1}{2}(0, i(e^z - e^{\bar {z}}) ,  e^z + e^{\bar {z}}) = (0, -e^u \sin v, e^u \cos v).
		\end{align*}
	Integrating, we see that
		\[
			x = (\tfrac{1}{2}H e^{2u}, e^u \cos v, e^u \sin v),
		\]
	giving us a sphere with radius $\frac{1}{H}$ by \eqref{eqn:sphereI3} (see Figure~\ref{fig:sphereI3}).
	
	In fact, if $h_2 = 0$, then the resulting surface is a sphere: Kenmotsu-type representation tells us that
		\[
			x_z \dif{z} = (\bar{h}_1, 1, -i) \omega, \quad x_{\bar{z}} \dif{\bar{z}} = (h_1, 1, i) \bar{\omega}
		\]
	Since we have that $h_1$ is holomorphic,
		\[
			\dif{h}_1 = h_{1,z} \dif{z} = H \omega,
		\]
	so that
		\[
			x = \frac{1}{2 H}\left(h_1 \bar{h}_1, 2 \Re h_1, 2 \Im h_1\right)
		\]
	which is a sphere with radius $\frac{1}{H}$ by \eqref{eqn:sphereI3}.

	\begin{figure}
		\centering
		\savebox{\largestimage}{\includegraphics[width=0.435\textwidth]{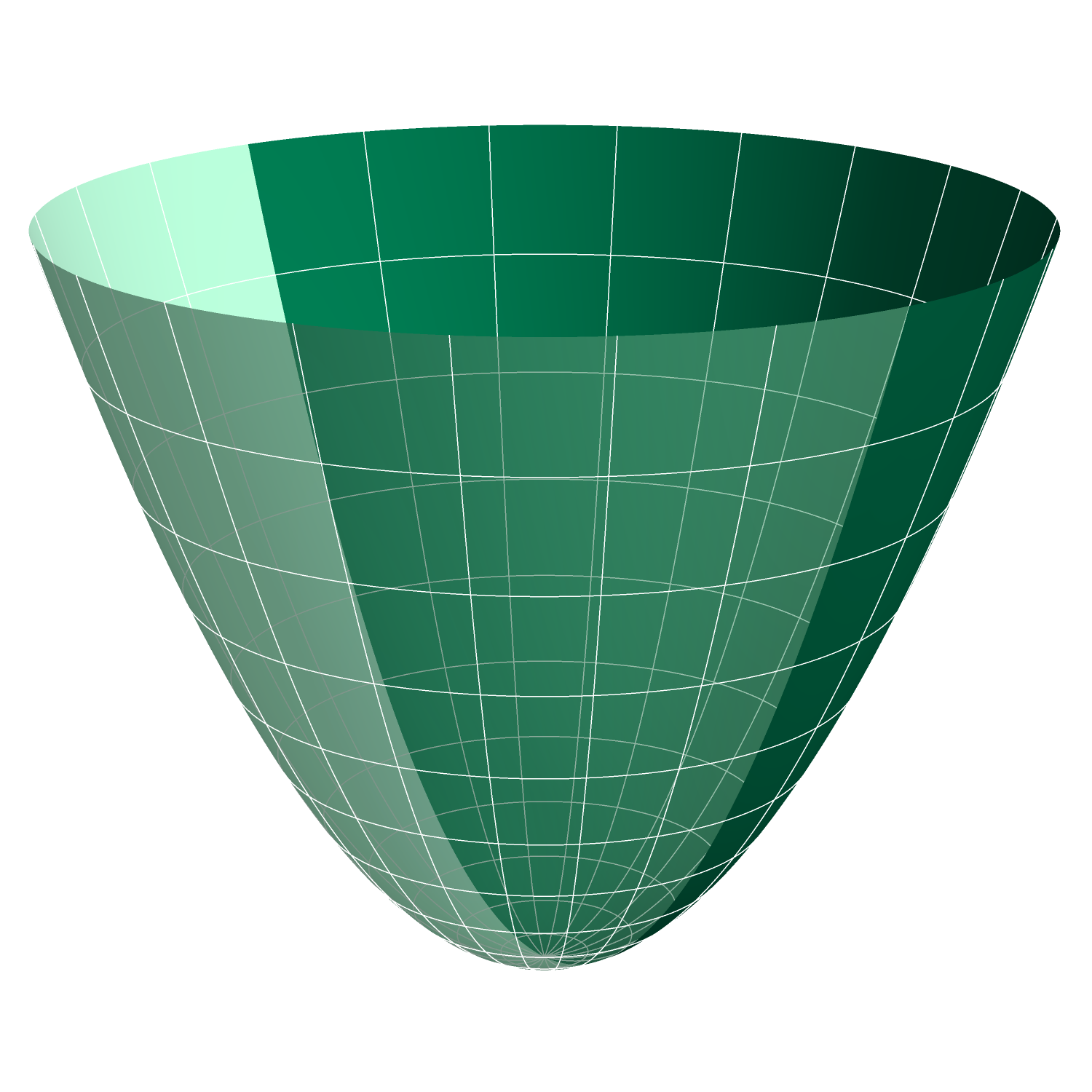}}%
		\begin{subfigure}[b]{0.435\textwidth}
			\centering
			\usebox{\largestimage}
			\caption{A sphere}
			\label{fig:sphereI3}
		\end{subfigure}
		\begin{subfigure}[b]{0.545\textwidth}
			\centering
			\raisebox{\dimexpr.5\ht\largestimage-.5\height}{\includegraphics[width=\textwidth]{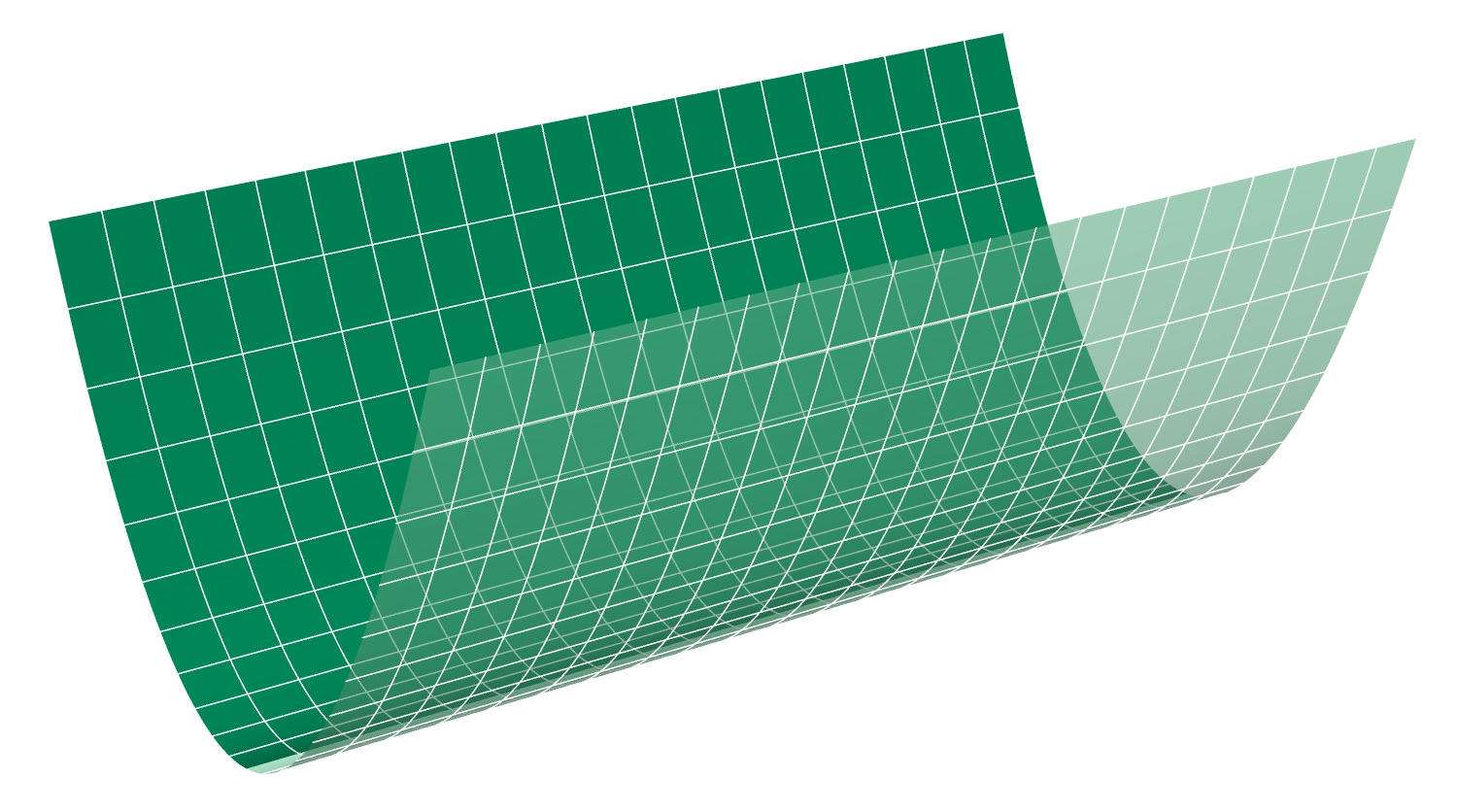}}
			\caption{A cylinder}
			\label{fig:cylinder}
		\end{subfigure}
		\caption{Trivial examples of cmc surfaces in isotropic $3$-space.}
		\label{fig:trivial}
	\end{figure}

	\item[Cylinders]
		Fixing some constant $H$, let $\omega = \dif{z}$ and $h_2 = H z$.
	Thus we have
		\[
			h_1 = H \int \dif{z} = Hz.
		\]
	Using the Kenmotsu-type representation, we find
		\[
			x_z = \frac{1}{2}(H(\bar{z} + z), 1, -i) =  \frac{1}{2}(2Hu,1,-i), 
				\quad x_{\bar{z}} =  \frac{1}{2}(2Hu,1,i).
		\]
	Calculating that
		\[
			x_u = x_z + x_{\bar{z}} = (2Hu,1,0), \quad x_v = i (x_z - x_{\bar{z}}) = (0,0,1),
		\]
	we can integrate and obtain
		\[
			x = (Hu^2, u, v).
		\]
	This is the cmc $H$ cylinder in the isotropic $3$-space (see Figure~\ref{fig:cylinder}).

	\item[Delaunay-type surfaces]
	For this example, let $\omega = e^z \dif{z}$ so that $h_1 = H e^z$, and set $h_2 = a e^{-z}$ for some constant $a$.
	Using the Kenmotsu-type representation, we find
		\begin{align*}
			x_z &= \frac{1}{2}\left(H e^{\bar{z}} + a e^{-z}, 1, -i\right)e^{z} = \frac{1}{2}\left(He^{2u} + a, e^{z}, -ie^{z}\right), \\
			x_{\bar{z}} &=  \frac{1}{2}\left(He^{2u} + a, e^{\bar{z}}, ie^{\bar{z}}\right).
		\end{align*}
	Therefore, we obtain
		\begin{align*}
			x_u &= x_z + x_{\bar{z}} = (He^{2u} + a,e^{u} \cos v, e^{u} \sin v), \\
			x_v &= i (x_z - x_{\bar{z}}) = (0,- e^{u}\sin v, e^{u} \cos v),
		\end{align*}
	and integrating tells us that
		\[
			x = \left(\frac{He^{2u} + 2au}{2},e^{u} \cos v, e^{u} \sin v\right).
		\]
	These are the rotationally invariant cmc $H$ surfaces in isotropic $3$-space (see Figure~\ref{fig:delaunay}).
	
	\begin{figure}
		\begin{minipage}{0.495\textwidth}
			\centering
			\includegraphics[width=0.9\textwidth]{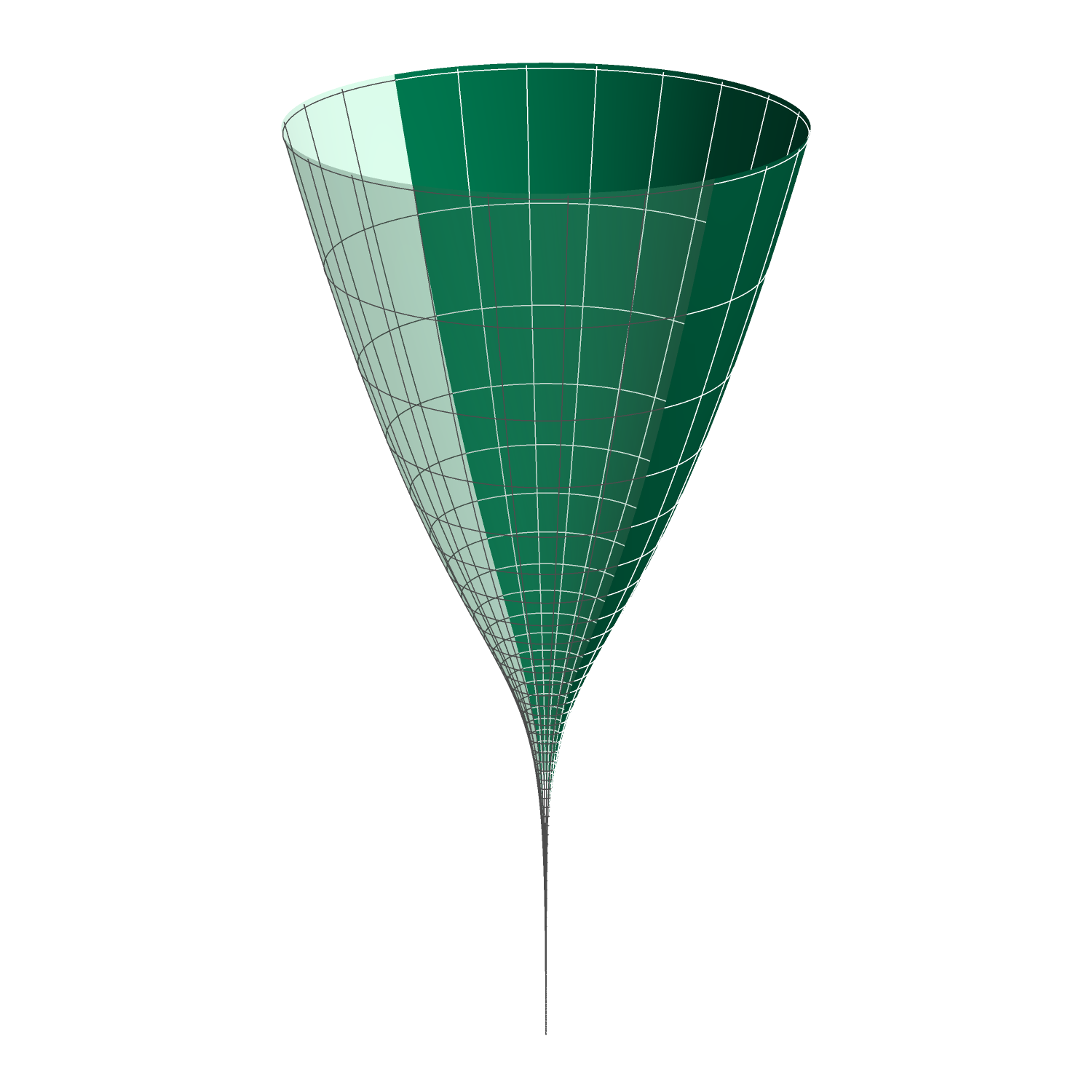}
		\end{minipage}
		\begin{minipage}{0.495\textwidth}
			\centering
			\includegraphics[width=0.9\textwidth]{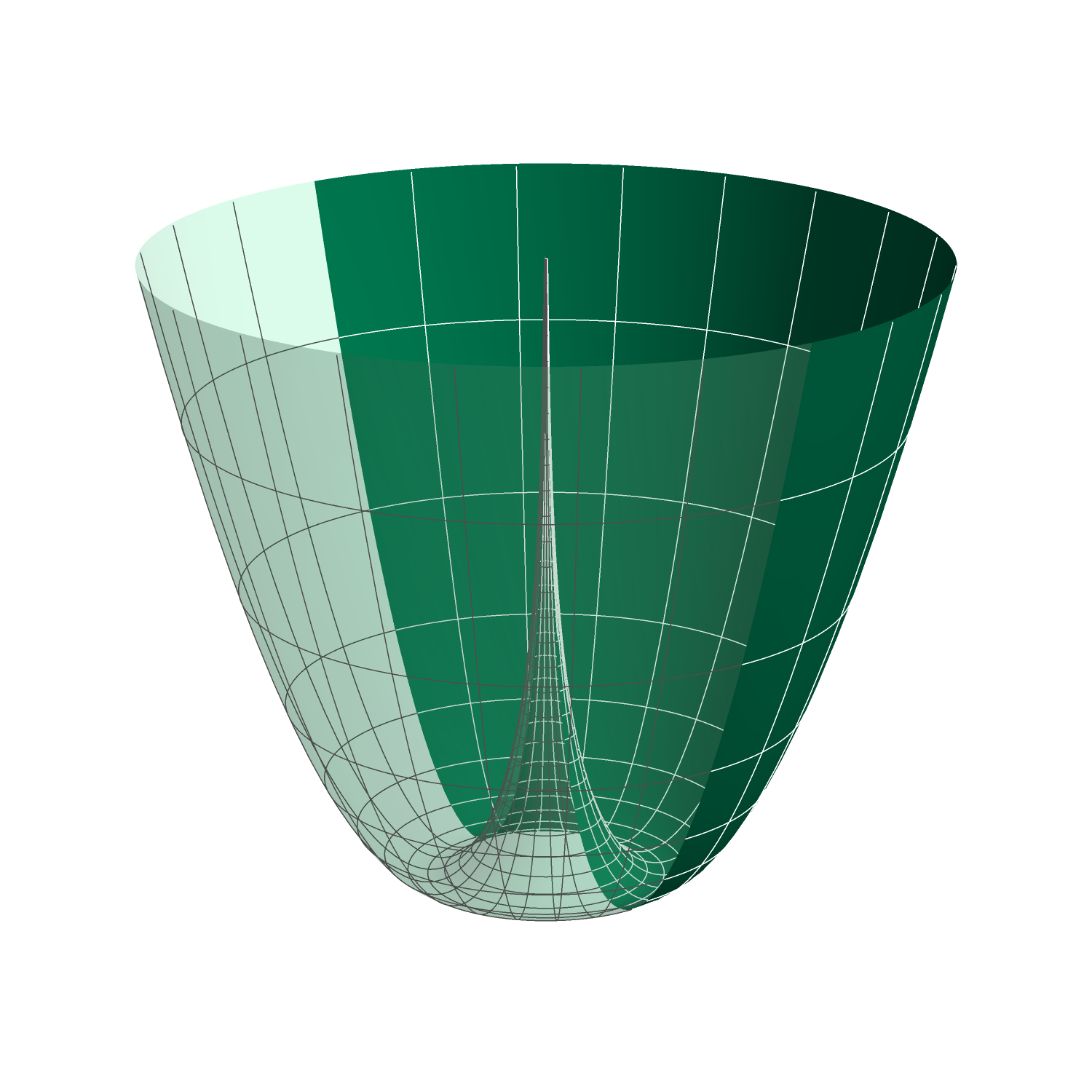}
		\end{minipage}
		\caption{Delaunay-type surfaces (with cmc-$1$) in isotropic $3$-space, drawn using $a=1$ (on the left), and $a = -2$ on the right.}
		\label{fig:delaunay}
	\end{figure}
	
	\begin{figure}
		\begin{minipage}{0.327\textwidth}
			\centering
			\includegraphics[width=\textwidth]{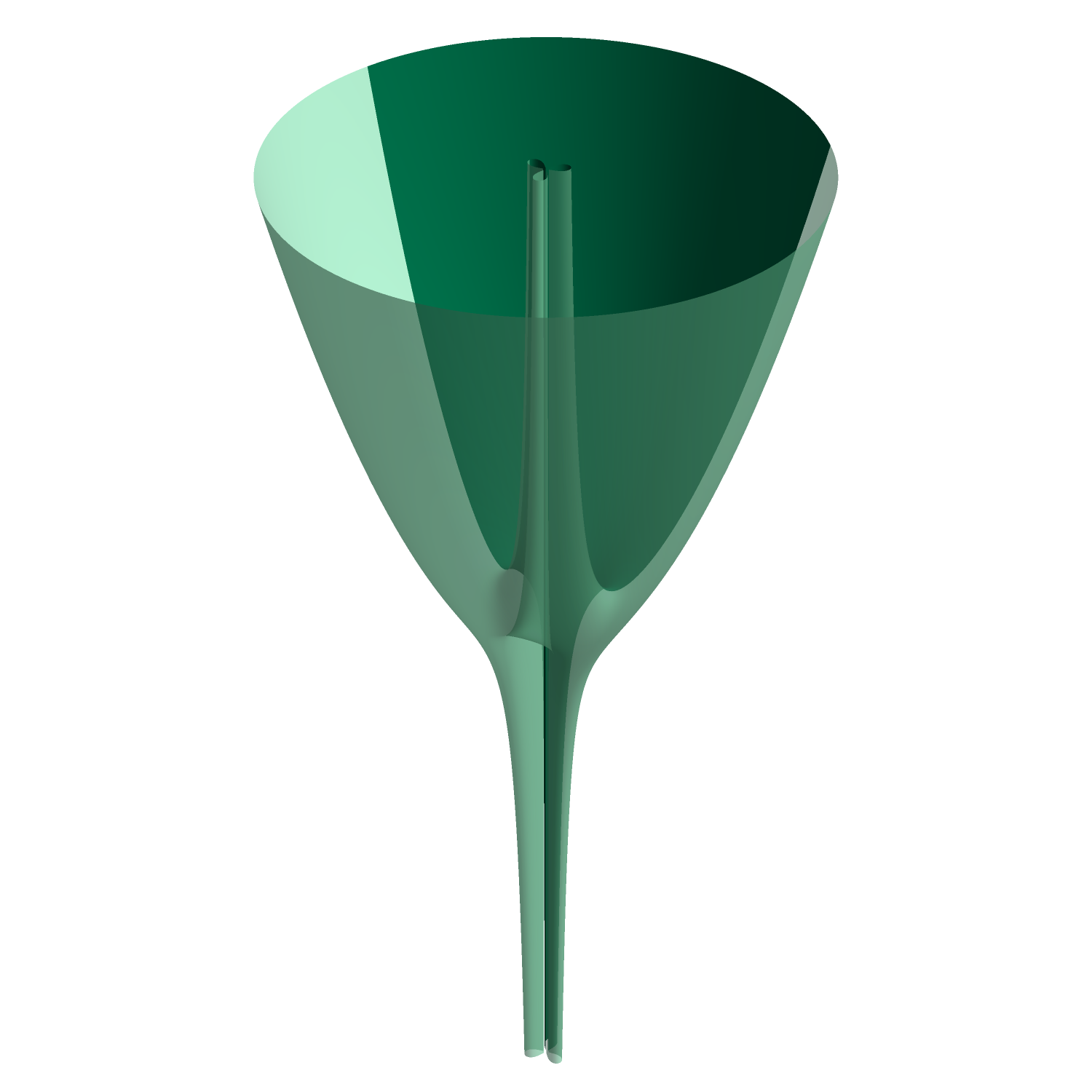}
		\end{minipage}
		\begin{minipage}{0.327\textwidth}
			\centering
			\includegraphics[width=\textwidth]{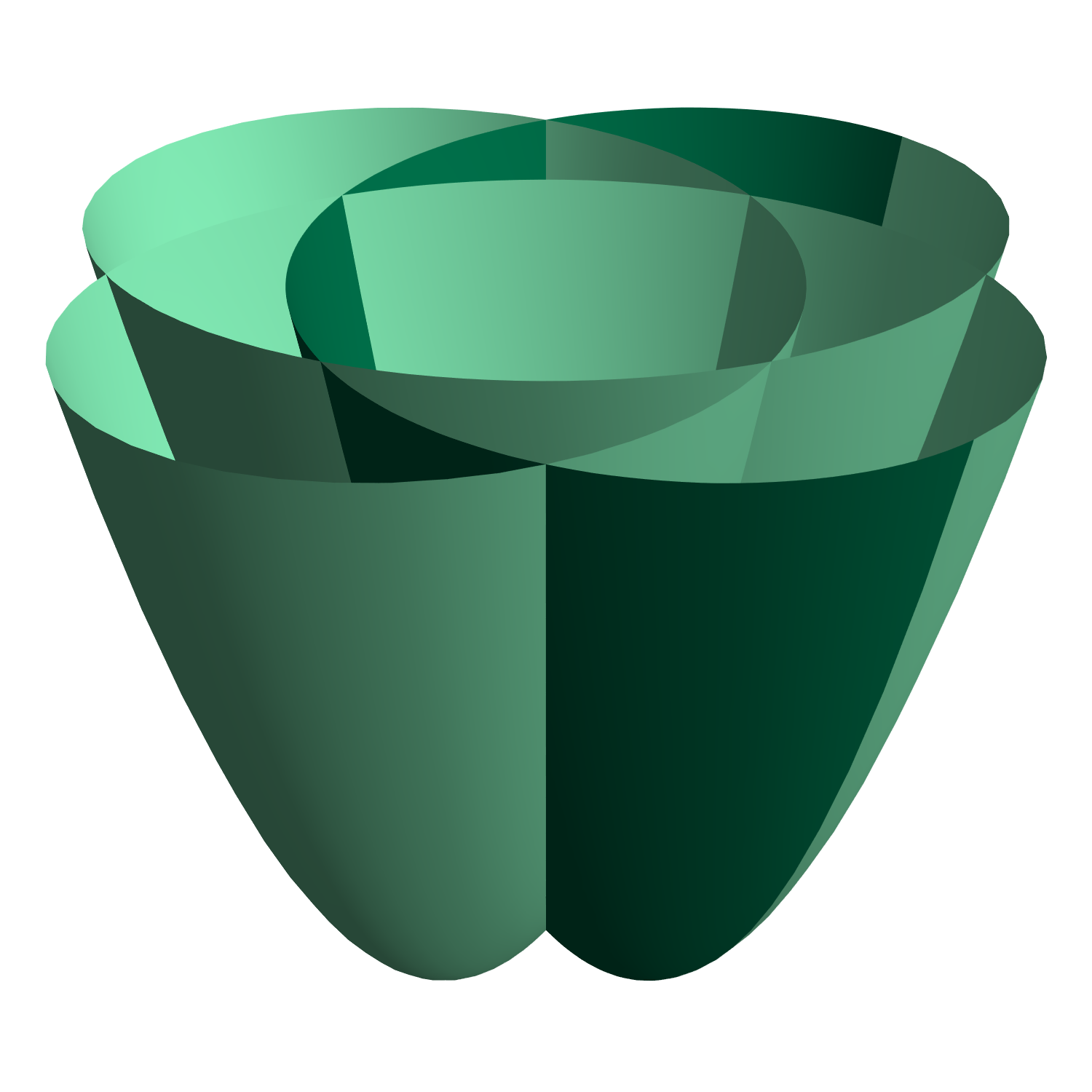}
		\end{minipage}
		\begin{minipage}{0.327\textwidth}
			\centering
			\includegraphics[width=\textwidth]{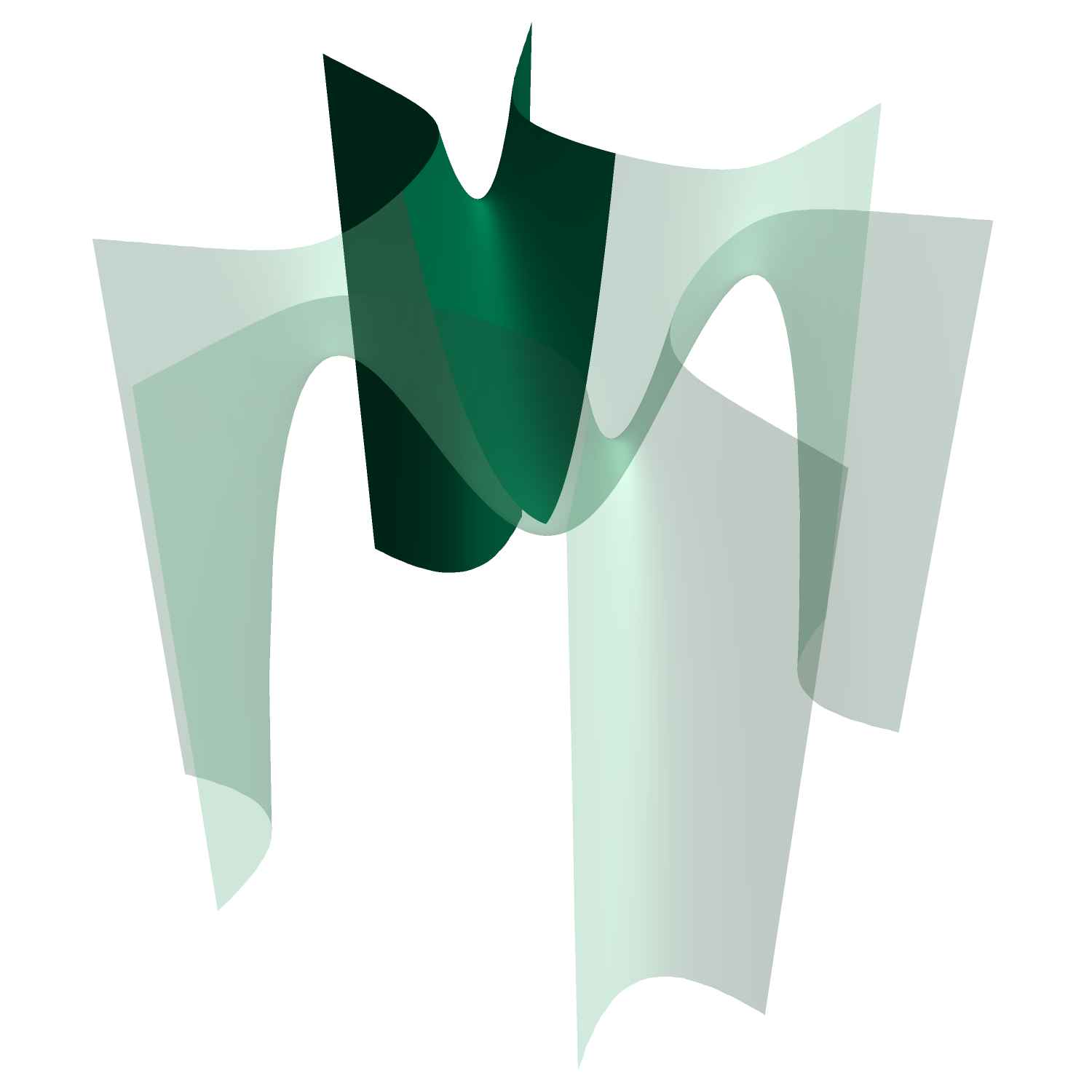}
		\end{minipage}
		\caption{Singly-periodic cmc-$1$ surfaces in isotropic $3$-space with dihedral symmetry, drawn using $a = 1, b = -3$ (on the left), $a = 2, b = \frac{8}{3}$ (in the middle), and $a = \frac{1}{3}, b = \frac{4}{3}$ (on the right).}
		\label{fig:single}
	\end{figure}
	
	\item[Singly periodic cmc surfaces]
	To obtain examples of singly periodic cmc surfaces, let $\omega = e^{az} \dif{z}$ so that $h_1 = \frac{H}{a} e^{az}$, and set $h_2 = e^{(b-a)z}$ for some constants $a,b$.
	Then the Kenmotsu-type representation tells us that
		\begin{align*}
			x_z &= \frac{1}{2}\left(\tfrac{H}{a} e^{a \bar{z}} + e^{(b-a)z}, 1, -i\right)e^{a z} = \frac{1}{2}\left(\tfrac{H}{a}e^{2au} + e^{bz}, e^{a z}, -ie^{a z}\right), \\
			x_{\bar{z}} &= \frac{1}{2}\left(\tfrac{H}{a}e^{2au} + e^{b\bar{z}}, e^{a \bar{z}}, ie^{a \bar{z}}\right).
		\end{align*}
	Thus we calculate that
		\begin{align*}
			x_u &= x_z + x_{\bar{z}} = \left(\tfrac{H}{a}e^{2au} + e^{bu}\cos bv ,e^{a u} \cos a v, e^{a u} \sin av\right), \\
			x_v &= i (x_z - x_{\bar{z}}) = \left(-e^{bu}\sin bv ,- e^{a u} \sin a v, e^{a u} \cos av\right).
		\end{align*}
	Integrating, we obtain
		\[
			x = \left(\frac{H e^{2au}}{2a^2} + \frac{e^{bu}\cos bv}{b} ,\frac{e^{a u}\cos a v}{a} , \frac{e^{a u} \sin av}{a} \right).
		\]

	Let $a, b$ be rational numbers so that there are relatively prime pair $(p_1, q_1), (p_2, q_2) \in \mathbb{Z}^2$ with
		\[
			a = \frac{p_1}{q_1}, \quad\text{and}\quad b = \frac{p_2}{q_2}.
		\]
	For such $a,b \in \mathbb{Q}$, the surface is $2\pi L$-periodic over $v$, i.e.\ $x(u,v) = x(u, v+ 2\pi L)$ for
		\[
			L = \frac{\lcm(q_1, q_2)}{\gcd(p_1, p_2)},
		\]
	where $\gcd$ and $\lcm$ stands for the greatest common divisor and the least common multiple, respectively.
	In particular, the cmc surface has dihedral symmetry $D_{L|b|}$ (see Figure~\ref{fig:single}).
\end{description}

\textbf{Acknowledgements.}
The last author gratefully acknowledges the support from NRF of Korea (2017R1E1A1A03070929 and 2020R1F1A1A01074585).
	

%
\end{document}